\documentclass[12pt]{article}
\usepackage{euscript,graphicx,epstopdf,amscd,amsgen,amsfonts,amssymb,latexsym,amsmath,amsthm,graphicx,mathrsfs,times,color,overpic,bbm, array}
\usepackage{euscript,graphicx,epstopdf,amscd,amsgen,amsfonts,amssymb,latexsym,amsmath,amsthm,graphicx,mathrsfs,times,color,overpic,fourier}
\usepackage{graphicx}
\usepackage{latexsym,amssymb}
\usepackage{amsthm}
\usepackage{indentfirst}
\usepackage{amsmath}
\usepackage{color}
\usepackage{xcolor}
\usepackage{fourier}
\usepackage[all]{xy}
\usepackage[colorlinks=true, linkcolor=blue, citecolor=blue]{hyperref}
 \usepackage{authblk}

\newtheorem{theoremalph}{Theorem}

\newtheorem*{Main Theorem}{Main Theorem}
\newtheorem{Coro}[theoremalph]{Corollary}
\newtheorem{Theorem}{Theorem}[section]
\newtheorem*{Theorem A}{Theorem A}
\newtheorem*{Theorem A'}{Theorem A'}
\newtheorem*{Theorem B'}{Theorem B'}

\newtheorem{Definition}[Theorem]{Definition}
\newtheorem{Proposition}[Theorem]{Proposition}
\newtheorem{Lemma}[Theorem]{Lemma}

\newtheorem{Question}{Question}

\newtheorem{Remark}[Theorem]{Remark}

\newtheorem{Claim-numbered}[Theorem]{Claim}
\newtheorem{SubClaim-numbered}[Theorem]{Subclaim}

\numberwithin{equation}{section}

 \def\NN{{\mathbb N}}

\def\TT{{\mathbb T}}

\def\Diff{\hbox{Diff} }
\def\Gib{{\mathrm{Gibb}^u}}

   \def\cO{{\cal O}} \def\cU{{\cal U}}
    \def\cV{{\cal V}}
    \def\cW{{\cal W}}
\def\cF{{\cal F}}

\newcommand{\eqdef}{\stackrel{\scriptscriptstyle\rm def}{=}}

\newcommand{\diff}{{\operatorname{Diff}}}

\def\diff{\operatorname{Diff}}
\def\ph{\operatorname{PH}}

\def\dim{\operatorname{dim}}

\def\supp{\operatorname{supp}}
\def\ud{\operatorname{d}}
\def\e{{\varepsilon}}

\def\det{\operatorname{det}}

\begin{document}

\title{A Conservative  Partially Hyperbolic Dichotomy:
Hyperbolicity versus  Nonhyperbolic Measures}

\author{Lorenzo J. D\'iaz, Jiagang Yang and  Jinhua Zhang\footnote{
	  				We are deeply grateful to 
				Sylvain Crovisier, 
				Rafael Potrie, and Amie Wilkinson for their numerous valuable insights 
				into the properties of partially hyperbolic dynamics, especially emphasizing 
				the importance of the transverse condition. Their comments on \cite{AviCroWil:,CroPot:,AviCroEskPotWilZha:} were particularly relevant. 
				
				L. J.  D\'iaz is partially supported by
				CAPES -- Finance Code 001, 
				CNPq-grant  310069/2020-3, 
				CNPq Projeto Universal 
				404943/2023-3, 
				INCT-FAPERJ E-26/200866/2018, and CNE-FAPERJ E26/204046/2024.  
				J. Yang is partially supported by CAPES -- Finance Code 001,  CNPq, FAPERJ, PRONEX, NSFC  12271538, NSFC 11871487,  NSFC 12071202 and MATH-AmSud 220029.
				J. Zhang is partially supported by  National Key R$\&$D Program of China (2021YFA1001900),  NSFC 12471179 and the Fundamental Research Funds for the Central Universities.  
				 
		 The authors thank the School of Mathematical Sciences of
				Beihang University for the warm hospitality  and the support while preparing this paper.}}


\maketitle

\begin{abstract}
	In a conservative and partially hyperbolic  three-dimensional setting,
	we study three representative classes of diffeomorphisms:
	those  homotopic to Anosov (or Derived from Anosov  diffeomorphisms),
	diffeomorphisms in neighborhoods of the time-one map of the geodesic flow on a surface of negative curvature, 
	and accessible and dynamically coherent skew products with circle fibers.
	In any of these classes, we establish the following dichotomy: either the diffeomorphism is Anosov, or it possesses nonhyperbolic ergodic measures. Our approach is perturbation-free and combines recent advances in the study of stably ergodic diffeomorphisms with a variation of the periodic approximation method to obtain ergodic measures.
	
	A key result in our construction, independent of conservative hypotheses, is the construction of nonhyperbolic ergodic measures for sets with a minimal strong unstable foliation that satisfy the mostly expanding property. This approach enables us to obtain nonhyperbolic ergodic measures in other contexts, including some subclasses of the so-called anomalous partially hyperbolic diffeomorphisms that are not dynamically coherent.
\end{abstract}
%

	

	\maketitle

	\tableofcontents

	\section{Introduction}

	\subsection{A question and our results}
	This paper addresses the following question:
	\begin{equation}\label{eqquestion}
		\text{How does ``nonhyperbolic behavior''  manifest in dynamical systems?}
	\end{equation}
	We explore this within a conservative and partially hyperbolic setting where the center bundle is one-dimensional.
	In this context, we can consider (whenever it exists) the Lyapunov exponent of points with respect to this direction, called {\em{center Lyapunov exponent.}} They provide a natural quantifier to investigate nonhyperbolicity.
	Note that the center bundle is one of the bundles of the Oseledets splitting for any ergodic measure.
	Hence, by Oseledets ergodic theorem \cite{Ose:68}, almost every point (with respect to some ergodic measure)
	has the same exponent.
	An ergodic measure is \emph{hyperbolic} if all its exponents are nonzero, otherwise, it is \emph{nonhyperbolic}.
	As in partially hyperbolic settings the exponents of ergodic  measure associated to bundles different from the center one  
	are either negative or positive, the absence of hyperbolicity of a measure is reflected by a zero Lyapunov exponent associated with the center bundle.

	We consider three  representative classes of
	partially hyperbolic  and conservative diffeomorphisms in dimension three:
	(a) homotopic to Anosov (or Derived from Anosov  diffeomorphisms),
	(b)  diffeomorphisms in a neighborhood of the time-one map of the geodesic flow on a surface of negative curvature, 
	and (c) accessible and dynamically coherent skew products with circle fibers.
	Although these three classes share some common properties, each class exhibits specific features that require a separate study. See Section~\ref{s.statements} for precise definitions and terminology.
	For a long time, these three classes were believed to represent the ``essential'' examples of conservative partially hyperbolic systems in dimension three. 
	However, new classes have since been introduced, with the so-called ``anomalous'' examples in \cite{Bonetal:20} being particularly relevant in our context (see the comments below and \cite[Section 1.2]{BonGovPot:16} for further discussion).
	For  every diffeomorphism in any of  these three classes,  we prove the following dichotomy: 
	\begin{itemize}
		\item
		either the diffeomorphism is Anosov, 
		\item
		or it has a nonhyperbolic measure,
	\end{itemize}
	as shown in Theorems~\ref{thm.DA-case}, ~\ref{thm.geodesic-flow-case}, and ~\ref{thm.nilmanifold-case}, which  deal with cases
	(a), (b), and (c) above, respectively. Note that in cases (b) and (c), only the second alternative of the dichotomy can occur.
	Note that the diffeomorphisms in these three theorems are dynamically coherent (see the next section for details). However, some of our methods do not rely on this condition. Having this in mind, we state 
	Theorem~\ref{t.transversecondition}
	(where no dynamical coherence is required)
	involving the so-called transversality condition introduced in
	\cite{AviCroWil:,CroPot:}. 
	This result also presents a dichotomy:  ``hyperbolicity versus  existence of nonhyperbolic measures''. 
	Though somewhat technical, it is presented in a way that makes it ready for practical use and holds potential for application in a variety of contexts, particularly in situations beyond coherence.
	
	The results above follow from the general (not necessarily conservative) Theorem~\ref{thm.existence}, which provides conditions, such as the minimality of the strong unstable foliation, the mostly expanding property, and the presence of saddles with different types of hyperbolicity, that imply a set supports nonhyperbolic measures.
	A consequence of this theorem is  Corollary~\ref{ctct.both-minimal}, which asserts that the stably ergodic and nondynamically coherent examples
	in \cite{Bonetal:20} have nonhyperbolic measures, see Remark~\ref{rem:bonetal}.
	Precise definitions and statements of our results can be found in Section~\ref{ss.statements}.
	
	\subsection{Context}
	Many problems in modern smooth Dynamical Systems Theory can be framed as variations of the question in
	\eqref{eqquestion}. In the 1960s, it was established that uniformly hyperbolic systems, such as Anosov diffeomorphisms and those satisfying Axiom A, are not dense within the space of dynamical systems, see \cite{AbrSma:68}. 
	Following this discovery, contributions from \cite{New:70,Shu:71,Man:78,Dia:95,BonDia:96} introduced different  mechanisms leading to persistent forms of nonhyperbolicity.
	These findings led to the need for a broader definition of hyperbolicity, resulting in  concepts like partial hyperbolicity and, notably, nonuniform hyperbolicity, as introduced in Pesin's theory, \cite{Pes:77}.
	These concepts are central to this paper. 
	Additionally, the need to explore dynamics beyond uniform hyperbolicity became evident. Although the existence of a nonhyperbolic measure is a sufficient condition for nonhyperbolicity, it is not a necessary one. There are numerous examples of diffeomorphisms where all ergodic measures are hyperbolic, yet the diffeomorphism itself is not, as illustrated in \cite{BalBonSch:99, CaoLuzRio:06,Diaetal:09}, among others. 
	The main results of this paper claim that for the three classes of diffeomorphisms mentioned above, the existence of  a nonhyperbolic measure is both a sufficient and a necessary condition.
	
	Several conjectures have been proposed to guide the study of dynamics beyond hyperbolicity. A well-known conjecture 
	by Palis \cite{Pal:00}  states the dichotomy ``hyperbolicity versus approximation by cycles''. This conjecture has been  reformulated in various ways, including the ``hyperbolicity versus robust cycles''  conjecture in \cite{Bon:11}  and the 
	``three alternatives''  proposed in \cite{Hay:14}, adding some ergodic elements. In \cite{DiaGor:09},  an ergodic local 
	version of Palis' conjecture is formulated: $C^1$-generically,
	a homoclinic class  (roughly an elementary piece of dynamics, see the discussion in \cite[Section 10.4]{BonDiaVia:05})
	is either uniformly hyperbolic or it supports a nonhyperbolic  measure. This conjecture was proved
	in \cite{Cheetal:19}.  However, without the genericity assumption, this conjecture does not hold true, as discussed above. 
	In this work, we address the conjecture from \cite{DiaGor:09} by focusing on the three classes of partially hyperbolic and conservative $C^{1+\alpha}$-diffeomorphisms above and establish a stronger version: a 
	diffeomorphism is either uniformly hyperbolic or supports a nonhyperbolic measure. 
	
	\subsection{Approach and ingredients}
	We now present the tools and techniques in this paper.
	Our approach combines a variation of the  {\em{GIKN-method of periodic approximations}} introduced in \cite{Goretal:05}  along with recent advances in the theory of three-dimensional partially hyperbolic and  conservative diffeomorphisms. 
	Additionally, we employ
	results developed to address a conjecture
	by Pugh and Shub \cite{PugShu:97} which claims that ``most diffeomorphisms among the set of conservative partially hyperbolic diffeomorphisms are stably ergodic''.
	These advances encompass a range of topics, including consequences of dynamical coherence, properties of strong stable and unstable foliations (such as accessibility), and  properties of measures of maximal entropy, among others. The types of results applied vary depending on the context and are detailed in Sections~\ref{ss.homotopic}, \ref{ss.strongacc}, and \ref{ss.maximalent}.
	General key underlying ingredients are the Gibbs $u$-states and  the mostly expanding property, see Section~\ref{ss.Gibbsmostly}.
	
	We introduce in Theorem~\ref{thm.existence} a general method 
	(where no conservative assumption is involved)
	for constructing nonhyperbolic measures built up on the ideas in \cite{Goretal:05}. 
	This result requires that there is a set which is saturated by the strong unstable foliation, has a periodic 
	point with negative center exponent, and has the mostly expanding property.
	This approach does not involve perturbations of the dynamics.
	We observe that
	Theorem~\ref{thm.existence} 
	can be also obtained in the $C^1$-topology considering systems satisfying the unstable entropy formula, see Remark~\ref{r.C1-u-Gibbs}. 
	
	The GIKN-method provides an ergodic measure with uncountable support, achieved as a suitable limit of ergodic measures supported on periodic orbits (called {\em{periodic measures}}). This method is particularly well-suited for partially hyperbolic settings where the center bundle is one-dimensional (this allows to recover the center Lyapunov exponent of a limit measure as the limit of the center Lyapunov exponents of the converging measures).
	In such settings, the strong stable and unstable foliations, which are always defined, play a fundamental role.
	
	The GIKN-method has been employed to construct nonhyperbolic measures in several contexts; see \cite{KleNal:07}  for some classes of diffeomorphisms and \cite{DiaGor:09, BonDiaGor:10, Cheetal:19,
		BonZha:19,WanZha:20} for nonhyperbolic homoclinic classes. 
	This method has since evolved invoking blender-horseshoes \cite{BocBonDia:16}, this approach allows to deal with a larger class of maps and provides sets supporting only nonhyperbolic measures. See also the developments in \cite{BonDiaBoc:18,BonDiaKwi:21}. 
	
	In a partially hyperbolic setting with a one-dimensional center bundle, all known methods 
	for obtaining nonhyperbolic measures
	typically involve a dynamical interaction between a  ``center contracting region'' and a  ``center expanding region''. In all  studied cases, the center contracting region is provided by a periodic point having a negative center Lyapunov exponent. The nature of the center expanding region varies by method: in \cite{Goretal:05}, it is a periodic point with a positive center Lyapunov exponent; in \cite{BocBonDia:16}, it is a center-expanding blender-horseshoe; and in this work, it is provided by the mostly expanding property. To ensure effective interaction between these regions, minimality-like properties of the 
	strong unstable foliations are used. In our approach, we work with sets that are saturated by the strong unstable leaves.
	
	Finally, we remark that our results consider $C^{1+\alpha}$-diffeomorphisms as this allows us to  use Pesin theory as well as $u$-Gibbs states and invoke the mostly expanding property.   
	
	This paper is organized as follows. In Section~\ref{s.statements}, 
	we introduce the terminology used
	throughout the paper and precisely state our main results.
	In Section~\ref{s.construction}, we present our method to construct nonhyperbolic
	ergodic measures, proving Theorem~\ref{thm.existence}. 
	In Section~\ref{s.hypversus}, employing Theorem~\ref{thm.existence}, we prove a semi-local version of our hyperbolicity dichotomy for 
	mostly expanding $\mathcal{F}^u$-minimal sets, see Proposition~\ref{p.intermediate}.
	In this section we also prove Theorem~\ref{t.transversecondition}.
	Sections~\ref{s.dicDA}, 
	~\ref{s.geodesicflow}, and ~\ref{s.skewcircle}
	are dedicated to the proofs of Theorems~\ref{thm.DA-case}, ~\ref{thm.geodesic-flow-case}, and ~\ref{thm.nilmanifold-case}, respectively.

	\section{Statements of the main results}
	\label{s.statements}
	
	\subsection{Preliminaries} \label{ss.preliminaries}
	We begin by introducing some notation and terminology.
	In what follows, $M$ denotes a   closed Riemannian manifold
	with a metric $\| \cdot\|$
	and 
	$\diff^r(M)$, $r\in \mathbb{N}$, denotes the space of $C^r$-diffeomorphisms endowed with the uniform topology.
	Given  $f\in\diff^1(M)$, an $f$-invariant compact set $\Lambda$  is \emph{partially hyperbolic} if there are a 
	$Df$-invariant continuous splitting $T_{\Lambda}M=E^s\oplus E^c\oplus E^u$ 
	with three nontrivial bundles
	and constants   $C>1$ and $\lambda\in(0,1)$ such that for every $x\in \Lambda$ and $n\in\mathbb{N}$, one has
	\begin{itemize}
		\item $\|Df^n|_{E^s(x)}\|<C\cdot\lambda^n \textrm{~~and~~}\|Df^{-n}|_{E^u(x)}\|<C\cdot\lambda^n;$
		\item $\|Df^n|_{E^s(x)}\|\cdot \|Df^{-n}|_{E^c(f^n(x))}\|<C\cdot\lambda^n\textrm{ and }\|Df^n|_{E^c(x)}\|\cdot \|Df^{-n}|_{E^u(f^n(x))}\|<C\cdot\lambda^n. $
	\end{itemize}
	Throughout this paper, we shall assume that $C=1$, this  can be obtained  by replacing the metric 
	by an equivalent one, see \cite{Gou:07}. 
	
	A diffeomorphism $f\in\diff^1(M)$ is \emph{partially hyperbolic} if the whole ambience $M$ is a partially hyperbolic set.  
	In this case, 
	it is classical that the distributions $E^s$ and $E^u$ are uniquely integrable to $f$-invariant foliations, called \emph{strong stable} and \emph{strong unstable foliations,} denoted by $\cF^s$ and $\cF^u$, respectively (see \cite{HirPugShu:77}). We denote by $\cF^*(x)$ the $\cF^*$-leaf through the point $x$, for $*=s,u.$  
	
	For $r \geq 1$, let $\ph^r (M)$ denote the set of $C^r$ partially hyperbolic diffeomorphisms on $M$, and let $\ph^r_{c=1}(M)$ represent the subset where $\dim(E^c) = 1$.
	It follows two key notions for partially hyperbolic diffeomorphisms.
	First, we say
	that $f\in \ph^r(M)$ is {\em{accessible}} if every pair of points in the manifold can be connected by a path consisting of finitely many arcs lying entirely in a strong stable or a strong unstable manifold.
	The accessible property is a $C^1$-open property and 
	$C^r$-dense in $\mathrm{PH}^r_{c=1}(M)$, where $r\ge 1$,  see \cite{Did:03} and \cite{Buretal:08}, respectively.
	See also the open and density results in \cite{DolWil:03} for
	$\mathrm{PH}^1(M)$.

	A diffeomorphism  $f\in \ph^r(M)$ is \emph{dynamically coherent}, if there exist $f$-invariant foliations $\mathcal{F}^{cs}$ and $\mathcal{F}^{cu}$ tangent to the bundles 
	$$
	E^{cs}  \eqdef E^s\oplus E^c \qquad \mbox{and} \qquad E^{cu}  \eqdef E^c\oplus E^u.
	$$ 
	In particular, taking the intersections of these two foliations, one gets a foliation $\mathcal{F}^c$ tangent to the center bundle $E^c$
	called the {\em{center foliation.}}

	Given $f\in \diff^1(M)$, we denote by
	$\mathfrak{M}_{\mathrm{inv}}(f)$ and
	$\mathfrak{M}_{\mathrm{erg}}(f)$ the  set of 
	invariant and
	ergodic measures of 
	$f$.
	Given any $f\in \ph^r_{c=1}(M)$ and  $\mu\in \mathfrak{M}_{\mathrm{erg}}(f)$, 
	note  that $E^c$ is always a direction of the Oseledets splitting of $\mu$.
	Note also that any other bundle of that splitting of is either contained in $E^s$ or $E^u$.
	
	The {\em{center Lyapunov exponent}} of a point $x\in M$ is defined by
	$$
	\lambda^c (x)  \eqdef \lim_{n\to \pm\infty} \frac{\log ||D_xf^n (v)||}{n}, \quad v\in E^c(x)\setminus \{\bar 0\},
	$$
	assuming that this limit exists.

	By the Birkhoff ergodic theorem, 
	given $f\in \ph^r_{c=1}(M)$,
	for every $\mu\in \mathfrak{M}_{\mathrm{erg}}(f)$  the \emph{center Lyapunov exponent} of $\mu$ is given by 
	\[
	\lambda^c(\mu) \eqdef\int\log\|Df|_{E^c}\|\ud\mu 
	\]
	and this number satisfies
	$$
	\lambda^c(\mu) = \lambda^c(x) \qquad \,\, \mu \,\, \mathrm{a.e.}\, x.
	$$
	For $\nu\in \mathfrak{M}_{\mathrm{inv}}(f)$, its {\em{averaged center Lyapunov exponent}} is defined  by  
	\begin{equation}
		\label{e.averageexp}
		\lambda^c(\nu) \eqdef\int\log\|Df|_{E^c}\|\ud\nu.
	\end{equation}

	Observe that Lyapunov exponents corresponding to the directions of the Oseledets splitting contained in $E^s$ 
	(respectively, $E^u$) are always negative (respectively, positive). In particular, the only Lyapunov
	exponent that may  be zero is the center one.
	
	\begin{Remark}
		\label{r.convergenceofexp}
		{\em{As the bundles of a partially hyperbolic depend continuously on the point
				(see, for instance, \cite[Chapter B1]{BonDiaVia:05}, the definition above implies that if $f\in \ph^r_{c=1}(M)$ and 
				$(\mu_n)_n$ is a sequence of measures in $\mathfrak{M}_{\mathrm{inv}}(f)$ with
				$\mu_n\to \mu$ in the $weak^\ast$-topology then
				$\lambda^c(\mu_n) \to \lambda^c(\mu)$.}}
	\end{Remark}

	A measure $\mu\in \mathfrak{M}_{\mathrm{erg}}(f)$ is  \emph{nonhyperbolic} if $\lambda^c(\mu)=0.$
	Otherwise the measure is called {\em{hyperbolic.}} We apply this terminology only to ergodic measures.

	\subsection{Main results} \label{ss.statements}
	To state our results, let
	$\diff_{\mathrm{Leb}}^r(M)$ be  the set of $C^r$-diffeomorphisms preserving a smooth volume $\mathrm{Leb}$
	and define
	$$
	\ph_{ \mathrm{Leb}, c=1}^r(M) \eqdef  \diff_{ \mathrm{Leb}}^r(M) \cap 
	\ph^r_{c=1}(M).
	$$

	We denote by $\mathbb{T}^3$ the three-torus.  Recall that a homeomorphism $f$ on $\mathbb{T}^3$ is \emph{homotopic to Anosov}  
	if the induced automorphism 
	$f_*:H_1(\mathbb{T}^3,\mathbb{Z})\to H_1(\mathbb{T}^3,\mathbb{Z})$, viewed as an element in $GL(3,\mathbb{Z})$, has no eigenvalues on the unit circle.

	\begin{theoremalph}\label{thm.DA-case}
		For every  $f\in\ph_{ \mathrm{Leb}, c=1}^{1+\alpha}(\mathbb{T}^3)$   homotopic to a linear  Anosov diffeomorphism
		the following dichotomy holds:
		\begin{itemize}
			\item either $f$ is Anosov;
			\item or $f$ has some nonhyperbolic ergodic measures.
		\end{itemize} 
	\end{theoremalph}
	Note that there exist conservative DA systems on $\mathbb{T}^3$ that are not Anosov, see for instance \cite{PonTah:14}, thus both alternatives in the dichotomy
	may occur.
	
	\begin{theoremalph}\label{thm.geodesic-flow-case}
		Let $(\phi_t)_{t\in\mathbb{R}}$ be the geodesic flow on a negatively curved surface $S$.    Then there exists a $C^1$-neighborhood $\cU\subset \diff^1_{\mathrm{Leb}}(T^1S)$ of $\phi_1$ such that every $C^{1+\alpha}$-diffeomorphism $f\in\cU$ has 
		a nonhyperbolic ergodic  measure.
	\end{theoremalph}
	\begin{Remark}
		{\rm Observe that $\mathbb{T}^3$ is the only three-manifold 
			supporting Anosov diffeomorphisms \cite{New:70b}, and the unit tangent bundle of a negatively curved surface is different from $\mathbb{T}^3$. Thus in this case, we only obtain nonhyperbolic ergodic measures. }
	\end{Remark}

	Finally,  we consider skew-products on a space which is a circle bundle and the  base space is a topological torus.
	Note that	Anosov diffeomorphisms does not have compact center leaves, thus the diffeomorphisms in Theorem~\ref{thm.nilmanifold-case} below can not be Anosov. The relevant point is that they have nonhyperbolic ergodic measures.
	
	Note that in the next theorem the ambience manifold may have any dimension greater than three.
	
	\begin{theoremalph}\label{thm.nilmanifold-case}
		Let 
		$M$ be a compact manifold without boundary with $\dim (M) \ge 3$ and
		$f\in\ph^{2}_{ \mathrm{Leb}, c=1}(M))$ be accessible and dynamically coherent. Assume that  the center foliation $\cF^c$ forms a circle bundle and the quotient space $M^c = M/\cF^c$ is a torus. Then $f$ has  a nonhyperbolic  ergodic  measure
	\end{theoremalph}
	
	When $M$ is a three-dimensional nilmanifold other than $\mathbb{T}^3$, it is known that every partially hyperbolic diffeomorphism on $M$ is both accessible and dynamically coherent whose  center foliation forms a circle bundle,
	see \cite{RodRodUre:08, HamPot:14}. 
	This leads to the following corollary of Theorem~\ref{thm.nilmanifold-case}.
	
	\begin{Coro}\label{cor.nilmanifold-case}
		Let $M$ be a three-dimensional nilmanifold other than $\mathbb{T}^3$. 
		Then every diffeomorphism $f\in\ph^{2}_{ \mathrm{Leb},c=1}(M)$ has a nonhyperbolic ergodic measure.  \end{Coro}

	To prove these results, we give a criterion (Theorem ~\ref{thm.existence}) for the existence of nonhyperbolic ergodic measures using the mostly expanding property that we proceed to describe. 
	
	\subsection{A general criterion for nonhyperbolic measures}
	\label{ss.criterion}
	
	Throughout this section, 
	$M$ is a compact boundaryless manifold with $\dim (M) \ge 3$.
	Given $f\in \ph^{1+\alpha} (M)$,
	a measure $\mu \in \mathfrak{M}_{\mathrm{inv}}(f)$  is a \emph{$u$-Gibbs state} if its conditional measures along the strong unstable manifolds are absolutely continuous with respect to the Lebesgue measure. 
	We denote  by $\Gib (f)$ the set of $u$-Gibbs states of $f$.
	The existence of $u$-Gibbs states  was firstly obtained by Pesin and Sinai \cite{PesSin:82}.

	A  diffeomorphism $f\in\ph^{1+\alpha}(M)$ is \emph{mostly expanding} if all center Lyapunov exponents of  every 
	$\mu \in \Gib (f)$  are positive \cite{AndVas:18,Yan:21}. One can analogously define the mostly expanding property for strong unstable laminated sets. 
	Let us observe that this notion
	can be generalized to the $C^1$-setting via the unstable entropy formula, see for instance \cite{Led:84,Yan:21,CroYanZha:20,HuaYanYan:20}. 
	This will allow to reformulate the theorem below in the 
	$C^1$-setting, see Remark~\ref{r.C1-u-Gibbs}. 
	
	We will review relevant properties of $u$-Gibbs states and mostly expanding diffeomorphisms in 
	Section~\ref{ss.Gibbsmostly}.

	Given a lamination $\mathcal{F}$,   
	a set $\Lambda$ is called \emph{$\cF$-saturated} if $\cF(x)\subset \Lambda$ for every $x\in\Lambda$. And $\Lambda$ is called \emph{$\mathcal{F}$-minimal} if  $\overline{\mathcal{F}(x)}=\Lambda$ for every 
	$x\in\Lambda$. In other words,  $\Lambda$ is the smallest closed $\mathcal{F}$-saturated set.  
	A foliation $\mathcal{F}$ is \emph{minimal} if  $\overline{\mathcal{F}(x)}=M$ for every $x$.

	\begin{theoremalph}\label{thm.existence}
		Let 
		$M$ be a compact boundaryless manifold with $\dim (M) \ge 3$,
		$f\in\diff^{1+\alpha}(M)$, and $\Lambda$ be an  $f$-invariant compact set. Assume  that 
		\begin{itemize}
			\item  $\Lambda$ is partially hyperbolic  with $\dim (E^c)=1$;  		
			\item $\Lambda$ is $\mathcal{F}^u$-minimal;
			\item   $\Lambda$ is mostly expanding;  and
			\item  $\Lambda$ contains  a periodic point with negative center Lyapunov exponent.
		\end{itemize}
		Then $f$ has a nonhyperbolic ergodic measure whose support is the whole $\Lambda$. 
	\end{theoremalph} 
	
	The nonhyperbolic measures are obtained using the so-called GIKN method introduced in \cite{Goretal:05}, as detailed in 
	Section~\ref{ss.GIKN}.

	Let us briefly discuss the minimality of the strong foliations.
	Minimality of strong foliations is closely related to  transitivity. Recall that a diffeomorphism  is \emph{transitive} if it 
	has a dense orbit.  
	Among the set of robustly transitive 
	partially hyperbolic diffeomorphisms with one dimensional center, there exists a $C^1$-open and dense subset in which either the strong stable foliation  or the strong unstable foliation is robustly minimal \cite{BonDiaUre:02,RodRodUre:07}. Furthermore, if the diffeomorphisms have a periodic and compact center leaf, then both strong foliations are robustly minimal, and this is the case for skew-products with circle fibers and perturbations of time-one map of geodesic flows on negatively curved surfaces considered in this paper. For robustly transitive derived from Anosov diffeomorphisms obtained by Ma\~n\'e \cite{Man:78}, one only knows the minimality of one strong foliation, see \cite{BonDiaUre:02,PujSam:06}. Some further modification of Ma\~n\'e's construction  to obtain the minimality of both strong foliations can be found in \cite{RodUreYan:22}. 
	A key argument in \cite{BonDiaUre:02} is the perturbations to get blenders. 
	This mechanism also holds in the conservative setting \cite{Rodetal:10}. 
	Other perturbation tools in \cite{BonDiaUre:02} are $C^1$-closing lemma  and $C^1$-connecting lemma which hold in conservative setting \cite{PugRob:83,WeXi:00}. 
	Thus the results in \cite{BonDiaUre:02,RodRodUre:07} also hold in the conservative setting (see also the discussion in \cite[Section 5.6]{Bonetal:20}). 
	As volume preserving partially hyperbolic $C^{1+\alpha}$-diffeomorphisms that are accessible and have one dimensional center bundle are stably ergodic \cite{BurWil:10,RodRodUre:08b}, with the observation above, among that set there exists $C^1$-open and dense subset in which at least one of the strong foliations is robustly minimal.  Thus it is natural, in our setting,  to consider the  minimal properties of  strong foliations.

	The next result is a consequence of Theorem~\ref{thm.existence}:

	\begin{Coro}\label{ctct.both-minimal}
		Let  $f\in\ph^{1+\alpha}_{\mathrm{Leb}, c=1}(M)$
		such that its strong stable and unstable foliations are both minimal. Then the following dichotomy holds:
		\begin{itemize}
			\item either $f$ is Anosov;
			\item or $f$ has nonhyperbolic ergodic measures.
		\end{itemize} 
	\end{Coro}

	\begin{Remark}[Anomalous partially hyperbolic diffeomorphisms]
		\label{rem:bonetal}
		{\em{The class of anomalous examples in \cite{Bonetal:20} is extensive, with several proposed subcategories featuring specific properties. Some of these examples are stably ergodic and have minimal strong foliations (see \cite[Theorem 1.1]{Bonetal:20} and the subsequent comments). As a consequence of Corollary~\ref{ctct.both-minimal}, such diffeomorphisms exhibit nonhyperbolic measures. The construction of these examples in \cite{Bonetal:20} involves the use of cone fields alongside a detailed analysis of global dynamics. Since this type of construction is still in its early stages, it is likely that new examples can be developed using these methods, potentially leading to further applications of Theorem~\ref{thm.existence}.} }
	\end{Remark}
	
	We presented several settings where the accessibility property holds. This is explicitly required in Theorem~\ref{thm.nilmanifold-case}. By Theorem~\ref{thm.both-usu-accessible-of-geodesic-flow}, it is also an implicit condition in Theorem~\ref{thm.geodesic-flow-case}. Furthermore, according to \cite{HamShi:21}, every non-Anosov partially hyperbolic diffeomorphism on $\mathbb{T}^3$ that is homotopic to an Anosov diffeomorphism is accessible. 
	Thus this condition is also implicit in Theorem~\ref{thm.DA-case}.
	However, our proofs rely on additional geometric-like properties beyond accessibility. This raises the question of whether accessibility alone would be sufficient to prove our results.

	\begin{Question}
		Let $M$ be a three-manifold and   $f\in\ph^{1+\alpha}_{\mathrm{Leb}, c=1}(M)$
		be  accessible. Does the following dichotomy hold 
		\begin{itemize}
			\item either $f$ is Anosov;
			\item or $f$ admits nonhyperbolic ergodic measures?
		\end{itemize}
	\end{Question}
	We believe the answer to this question is affirmative. Applying the arguments from \cite{AviCroWil:,CroPot:} may be a fruitful strategy 
	to address this question.

	We conclude this section with Theorem~\ref{t.transversecondition}, which presents a somewhat technical dichotomy but provides a ready-to-use solution. It is related to the previous question and involves the "transverse condition" introduced in \cite{AviCroWil:} (see also \cite{CroPot:}). For further details, including the precise definition see the discussion at the beginning of Section \ref{ss.transversecondition}. 
	This is an elaborate and deep condition. In simple terms, it requires an abundance of nearby points whose local strong unstable manifolds are not ``aligned''  with the strong stable foliation. Besides its intrinsic interest, this theorem also serves 
	as an important technical step in our constructions, playing a fundamental role in the
	study of diffeomorphisms homotopic to Anosov.

	\begin{theoremalph}\label{t.transversecondition}
		Let  $f\in\ph^{1+\alpha}_{\mathrm{Leb}, c=1}(M)$ be accessible and such that $\lambda^c(\mathrm{Leb})>0$.
		Suppose that there exists a minimal $\cF^u$-saturated {and $f$-invariant}  set $\Lambda$  which is {$s$-transverse}.
		Then 
		\begin{itemize}
			\item either $f$ is Anosov;
			\item or $f$ has a nonhyperbolic ergodic measure supported on $\Lambda$.
		\end{itemize} 
	\end{theoremalph}
	
	A motivation for Theorem~\ref{t.transversecondition} comes from the following borrowed result, which is a direct consequence of \cite[Corollary 1.7]{AviCroWil:}.
	
	\begin{Theorem}[\cite{AviCroWil:}]
		\label{thm.acw}
		Let $f\in\ph^{1+\alpha}_{\mathrm{Leb}, c=1}(M)$ be accessible.
		Let $\Lambda$ be    a $\mathcal{F}^{u}$-saturated which is $f$-invariant and has  uniformly  expanding center.
		If $\Lambda$ is $s$-transverse, 
		then $\mathcal{F}^{u}$ is minimal. In particular, the diffeomorphism $f$ is Anosov.
	\end{Theorem}

	In view of Theorem~\ref{thm.acw}, to prove Theorem~\ref{t.transversecondition}, we need to explore dynamics where
	the transverse condition is not satisfied. To do so, we use Theorem~\ref{t.acepwz} borrowed from
	\cite{AviCroEskPotWilZha:}.

	\section{Construction of nonhyperbolic measures: Proof of Theorem~\ref{thm.existence}}
	\label{s.construction}
	
	In this section, we firstly recall  the GIKN criterion stating that  a sequence of periodic measures supported on orbits with the good approximation property would converge  to an ergodic measure. In Section~\ref{ss.shadowing},  we recall Liao's shadowing lemma which allows us to find periodic orbits without doing perturbations. Then, in Section~\ref{ss.Gibbsmostly}, we collect some basic properties for mostly expanding diffeomorphisms as well as mostly expanding sets.  In Section~\ref{ss.goodperiodic}, we combine these tools to
	develop a mechanism for obtaining periodic orbits with the good approximation property.
	This allows us to inductively construct a sequence of periodic measures satisfying the GIKN criterion, thereby proving Theorem~\ref{thm.existence},
	see Section~\ref{ss.endproofE}.
	
	\subsection{The GIKN criterion}
	\label{ss.GIKN}
	
	The nonhyperbolic measures will be obtained as limits of sequences of periodic measures that satisfy the GIKN criterion, introduced by Gorodetski et al. in \cite{Goretal:05}, which ensures the ergodicity of the limit of such sequences. 
	
	To proceed, we need to introduce some terminology.
	Given a compact metric space $(X,\ud)$
	and a dense subset  $\{\varphi_n\}$ of the space of continuous functions on $X$ endowed with $C^0$-norm
	$\|\cdot\|_0$,
	the distance of two   Borel probability measures $\mu,\nu$ on $X$ associated to that sequence
	is defined as 
	\[
	\ud(\mu,\nu) ) \eqdef \sum_{n=1}^\infty\frac{|\int\varphi_n\ud\mu-\int\varphi_n\ud\nu|}{2^n(\|\varphi_n\|_{0}+1)}.
	\]
	Independently of choice of the dense sequence of maps, 
	this metric gives rise to the weak$*$-topology on the space of Borel probability measures on $X$.

	Given a diffeomorphism $f$, denote by $\mathrm{Per}(f)$ the set of periodic points of
	$f$ and by $\mathrm{Per}_{\mathrm{hyp}}(f)$ its subset of hyperbolic ones.
	Given a periodic orbit $\cO$, the \emph{periodic measure} $\delta_{\cO}$ is the atomic probability measure  
	equidistributed  on the orbit $\cO$; in a formula and denoting by $\# S$ the cardinality of a set $S$,
	\[
	\delta_{\cO}  \eqdef \frac{1}{\#\cO}\sum_{x\in\cO}\delta_x. 
	\]
	Given a map $f$ and a point $q$, we denote by $\cO_q$ its orbit by $f$.

	\begin{Definition}
		Given a compact metric space $(X,\ud)$, a homeomorphism $f:X \to X$,
		a periodic orbit $\cO_q$ of $f$,   and two numbers $\e>0$ and $\kappa\in(0,1)$, a periodic orbit $\cO_p$ is \emph{$(\e,\kappa)$-good} for $\cO_q$ if there exist a subset $K\subset \cO_p$ and a surjective map $\pi:K\to\cO_q$ such that 
		\begin{itemize}
			\item $\# K/\#\cO_p\geq \kappa;$
			\item for every $x\in K$, one has \[\ud\big(f^i(x),f^i(\pi(x))\big)<\e\quad \textrm{for any $i=0,\cdots, m-1$},\] where $m$ is the period of $q$;
			\item  the cardinality $\#\pi^{-1}(y)$ is independent of the point $y\in\cO_q$. 
		\end{itemize}
	\end{Definition}

	We now recall the GIKN criterion introduced in \cite{Goretal:05}. 
	We state the slightly improved version in \cite[Lemma2.5]{BonDiaGor:10} that
	characterizes the support of the limit measures.

	\begin{Theorem}[GIKN criterion]~\label{thm.GIKN-criterion}
		Let   $f\colon X\to X$ be a homeomorphism on  a compact metric space $X$, $\cO_{p_n}$ a sequence of periodic orbits of $f$ whose periods  tend to infinity, and sequences of positive real numbers
		$(\e_n)$ and $(\kappa_n)$ such that:
		\begin{itemize}
			\item $\cO_{p_{n+1}}$ is $(\e_n,\kappa_n)$-good for the periodic orbit $\cO_{p_n}$; 
			\item $\sum_{n\in\NN}\e_n<+\infty$ and $\prod_{n\in\NN}\kappa_n\in(0,1)$.
		\end{itemize}
		Then the sequence of periodic measure $(\delta_{\cO_{p_n}})$ converges in the weak$*$-topology to an ergodic measure $\nu$ with support 
		$$
		\supp\nu=\bigcap_{n=1}^\infty\overline{\bigcup_{k=n}^\infty\cO_{p_k}}.
		$$
	\end{Theorem}

	\subsection{Liao's  shadowing lemma}\label{ss.shadowing}
	We recall the shadowing lemma by Liao~\cite{Lia:79},  stating its reformulation in ~\cite{Gan:02}.
	We will use this lemma to find periodic orbits satisfying the GIKN criterion.
	This lemma involves the concept of a dominated splitting, see \cite{BonDiaVia:05} for a definition. We skip this definition, for our purposes it is enough to have in mind that in our setting $E^{cs} \oplus E^u$ and $E^s\oplus E^{cu}$
	are both dominated splittings.  
	
	\begin{Theorem}[Theorem 1.1 in \cite{Gan:02}]~\label{thm.Liao-Gan-shadowing}
		Let $f\in\diff^1(M)$ and $\Lambda$ be an $f$-invariant compact set with a dominated splitting 
		of the form $T_\Lambda M=E\oplus F.$ Then for every $\lambda<0$, there exist  constants $L, d_0>0$ such that for any $d\in(0,d_0)$ and any orbit segment  $\{x,f(x),\cdots, f^n(x) \}$ such that
		\begin{enumerate}
			\item  $\ud(x,f^n(x))<d$; 
			\item $\frac{1}{k}\sum_{i=0}^{k-1}\log\|Df|_{E(f^i(x))}\|\leq \lambda   \textrm{ for every $1\leq k\leq n$};$ and
			\item $\frac{1}{k}\sum_{i=0}^{k-1}\log\|Df^{-1}|_{F(f^{n-i}(x))}\|\leq \lambda \textrm{ for every  $1\leq k\leq n$};$
		\end{enumerate}
		there exists a hyperbolic periodic point $p$ of period $n$ whose stable bundle has dimension $\dim (E)$
		such that 
		$$
		\ud(f^i(x),f^i(p))\leq L\cdot d, \quad \mbox{for every $i=0,\cdots,n-1$}.
		$$
	\end{Theorem}
	
	\begin{Remark}\label{r.liaogan}
		{\rm{The  sizes of the stable and unstable manifolds of the hyperbolic periodic orbits obtained in Theorem~\ref{thm.Liao-Gan-shadowing} have uniform size, depending only on $\lambda$. The stable and unstable manifolds are tangent to 
				(extended to a neighborhood of $\Lambda$)
				cone fields with respect to the bundles $E$ and $F$, respectively.
				For the  proofs of these assertions see  \cite[Corollary 3.3]{PujSam:06} and \cite[ Section 8.2]{AbdBonCro:11}.}}
	\end{Remark}
	
	\begin{Remark}\label{r.liaoganapplication}
		{\rm In this paper we will apply Theorem \ref{thm.Liao-Gan-shadowing} to  $f\in\ph^1_{c=1}(M)$. In this case $E=E^s$,  $F=E^c\oplus E^u$ or $E=E^s\oplus E^c$, $F=E^u$.  As $E^s$ and $E^u$ are uniformly contracting and expanding, thus we only need to check the condition for  the bundle $E^c$. 
		}
	\end{Remark}
	
	\subsection{$u$-Gibbs states and mostly expanding diffeomorphisms}
	\label{ss.Gibbsmostly}
	In this section, we recall some basic properties and results about $u$-Gibbs states and mostly expanding diffeomorphisms, which will be used 
	throughout the paper. Recall these definitions
	in the beginning of Section~\ref{ss.criterion}.
	
	Let us recall that an {\em{empirical measure}} is a probability measure of the form
	\begin{equation}
		\label{e.empirical}
		\epsilon_n(x) 
		\eqdef
		\frac{1}{n}\sum_{i=0}^{n-1}\delta_{f^i(x)}, \qquad {\mbox{$x\in M$ and $n\in \mathbb{N}$}.}  
	\end{equation}
	The $u$-Gibbs states are obtained as   the limit measures of  empirical measures of Lebesgue typical points  \cite{PesSin:82,BonDiaVia:05,Dol:01,CroYanZha:20,HuaYanYan:20}.

	\begin{Theorem}\label{thm.empirical}
		Let $f\in\ph^{1+\alpha}(M)$. Then for every $x\in M$ and for Lebesgue almost every point $y\in\cF^u(x)$, each 
		weak$\ast$-limit of the sequence of empirical measures $(\epsilon_n(y))$ is a u-Gibbs state.  
	\end{Theorem}
	
	\begin{Remark}
		{\em{	Theorem~\ref{thm.empirical} implies that every $\mathcal{F}^u$-saturated and $f$-invariant compact set has u-Gibbs states. }}
	\end{Remark}

	\begin{Remark}\label{r.C1-u-Gibbs}
		\em{Theorem \ref{thm.empirical} holds in $C^1$-setting if one replaces the $u$-Gibbs state by measures satisfying the unstable  entropy formula 
			$$
			h_{\mu}(f,\cF^u)=\int\log\det(Df|_{E^u})\ud\mu,
			$$ 
			See Theorem C in \cite{CroYanZha:20} and Theorem A in \cite{HuaYanYan:20}.}
		We observe that in the $C^1$ setting, any weak$\ast$-limit of empirical measures satisfies the unstable entropy formula.
	\end{Remark}

	The following result states some basic properties of $u$-Gibbs states whose proof can be found in    \cite[Section 11.2]{BonDiaVia:05}.
	\begin{Lemma}\label{l.basic-u-gibbs}
		Let $f\in\ph^{1+\alpha}(M)$ and $\mu\in \Gib(f)$. Then 
		\begin{itemize}
			\item  every ergodic component of $\mu$ is also a u-Gibbs state;
			\item the support of $\mu$ is saturated by strong unstable leaves.
		\end{itemize}
	\end{Lemma}
	
	The mostly expanding property is in fact a $C^1$-open property in the following sense: if $f\in \ph^{1+\alpha}(M)$ is mostly expanding then there is a $C^1$-neighborhood $\mathcal{U}$ of $f$ such that every $g\in \mathcal{U} \cap
	\Diff^{1+\alpha}(M)$ is partially hyperbolic and  mostly 
	expanding.  

	\begin{Proposition}[\cite{AndVas:18,Yan:21}]\label{p.bound-of-LE}
		Let $f\in\ph^{1+\alpha}(M)$ be a mostly expanding diffeomorphism. Then there exist a $C^1$-neighborhood $\cU$ of $f$, a constant $c>0$, and an integer  $n_0\in\mathbb{N}$ such that 
		\[
		\int\log\|Dg^{-n_0}|_{E^c_g}\|\ud\mu<-c  
		\quad 
		\mbox{for every $g\in\cU\cap \diff^{1+\alpha}(M)$ and $\mu\in \Gib (g)$},
		\]
		where $E^c_g$ denotes the center bundle of $g$.
	\end{Proposition}

	We close this section with a result that will be used in Sections~\ref{ss.proofda},~\ref{ss.thm.geodesic-flow-case}, and~\ref{ss.proofsk}.

	\begin{Theorem}[Theorem D in \cite{Yan:21}]~\label{thm.mostly-expanding-criterion}
		Suppose that $f\in\ph_{\mathrm{Leb}, c=1}^{1+\alpha}(M)$ is accessible
		and satisfies $\int\log\|Df|_{E^c}\|\ud{\mathrm{Leb}}>0$. Then $f$ is mostly expanding.
	\end{Theorem}

	\subsection{Good periodic approximations}
	\label{ss.goodperiodic}
	
	Next result is a key proposition to prove Theorem~\ref{thm.existence}, it allows to find a sequence of periodic orbits satisfying the GIKN criterion. We emphasize that the periodic orbits are obtained by 
	Theorem~\ref{thm.Liao-Gan-shadowing} and there is no  perturbation involved. 
	We essentially follow \cite{Goretal:05}, replacing the existence of a periodic orbit with positive center Lyapunov exponent
	by the mostly expanding property. 
	Somewhat similar ideas can be found in \cite{BonZha:19}.
	
	Given $p\in \mathrm{Per}(f)$ in a partially hyperbolic set with one-dimensional center,
	we denote by $\lambda^c(p)$ the Lyapunov exponent of $\delta_{\mathcal{O}_p}$ associated to the center direction.
	
	\begin{Proposition}\label{main.prop}
		Let $f\in\diff^{1+\alpha}(M)$  and $\Lambda$ be an invariant  compact set such that 
		\begin{itemize}
			\item  $\Lambda$ is partially hyperbolic with $\dim(E^c)=1$;
			\item  $\Lambda$ is a $\mathcal{F}^u$-minimal; 
			\item   $\Lambda$ is mostly expanding;  and
			\item  $\Lambda$ contains some $p\in\mathrm{Per}(f)$ with $\lambda^c(p)<0$.
		\end{itemize}
		Then there exist $\rho>0$ and $\eta\in(0,1)$ such that for every $q\in\Lambda\cap \mathrm{Per}(f)$ with $\lambda^c(q)<0$ and for every $\e>0$, there is $p\in\Lambda\cap \mathrm{Per}(f)$ such that 	
		\begin{itemize}
			\item $\eta\cdot \lambda^c(q)<\lambda^c(p)<0;$
			\item the orbit $\cO_p$ is $\big(\e, 1-\rho\cdot |\lambda^c(q)|\big)$-good for $\cO_q$.
		\end{itemize}
	\end{Proposition}
	
	Before proving this proposition, we recall some facts about homoclinic classes.
	Given a diffeomorphism $f$, two  points $p, q\in  \mathrm{Per}_{\mathrm{hyp}}(f)$  are 
	\emph{homoclinically related} if the stable manifold of the orbit $\cO_p$ has non-empty transverse intersection with the unstable manifold of the orbit $\cO_q$, and the stable manifold of $\cO_q$ has non-empty transverse intersection with the unstable manifold of $\cO_p.$ The \emph{homoclinic class}  of $p$ is defined as 
	\[
	H(p,f) \eqdef \mathrm{closure} \Big( \big\{\textrm{$q\in \mathrm{Per}_{\mathrm{hyp}}(f)$: $q$ is homoclinically related with $p$} \big\} \Big). 
	\]
	The homoclinic class of $p$ can alternatively defined as closure of the transverse intersection
	of the invariant manifolds of the orbit of $p$. The following is an standard property of homoclinic classes in our setting:
	
	\begin{Remark}\label{r.good-approximation-in-homoclinic-class}
		Let $f\in\ph^1_{c=1}(M)$  and  $H(p,f)$ be a homoclinic class. Then for any $\varepsilon_1>0$, $\varepsilon_2>0$, $\varepsilon_3>0$, 
		and $\kappa\in(0,1)$ and any periodic point $q$ homoclinically  related to $p$, there exists a periodic point $q^\prime$ homoclinically related to $p$ such that  
		\begin{itemize}
			\item $\cO_{q^\prime}$ is $(\varepsilon_1,\kappa)$-good for $\cO_q$; 			
			\item $\cO_{q^\prime}$  is $\varepsilon_2$-dense in $H(p,f)$; 
			\item $|\lambda^c(q)-\lambda^c(q^\prime)|<\varepsilon_3.$ 
		\end{itemize} 
	\end{Remark}
	
	Consider $\Lambda$ as in Proposition~\ref{main.prop}.
	Next remark is a standard consequence of the fact that $\Lambda$ is $\mathcal{F}^u$-minimal 
	and partial hyperbolic. Note the latter  hypothesis implies the existence of invariant 
	cone fields defined on a neighborhood of $\Lambda$.

	\begin{Remark}
		\label{r.intersectionstun}
		{\em{Under the hypotheses of Proposition~\ref{main.prop}, as the angle between $E^{cs}$ and $E^u$ is uniformly bounded  from below, there are $r_0>0$ and $Df^{-1}$-invariant cone field $\mathcal{C}^{cs}$ around $E^{cs}$,
				defined on a neighborhood of $\Lambda$, 
				such that for every $r\in (0, r_0)$ there is $\delta>0$ 
				such that for every $x\in \Lambda$ the leaf $\mathcal{F}^u(x)$ transversely intersects
				any local manifold tangent to $\mathcal{C}^{cs}$
				of inner radius $r$  
				centered at any point
				$\delta$-close to
				$\Lambda$.}}
	\end{Remark}

	\begin{Lemma}
		\label{l.homorelated}
		Under the hypotheses of Proposition~\ref{main.prop}, every pair of periodic points  $p,q\in \Lambda$ with negative Lyapunov exponents
		are homoclinically related.
	\end{Lemma}
	
	\begin{proof}
		Note that  $\lambda^c(p)<0$ and $\dim (E^c)=1$ imply that
		$W^u (p) = \mathcal{F}^u (p)$ and hence it is dense in $\Lambda$. 
		By Remark~\ref{r.intersectionstun},
		one has that $W^u(p)$ transversely
		intersects $W^s(q)$. Changing the roles of $p$ and $q$ we get that $W^u(q)$ transversely intersects $W^s(p)$, proving that the periodic points are
		homoclinically related.
	\end{proof}

	\begin{Lemma}
		\label{l.homoclass}
		Under the hypotheses of Proposition~\ref{main.prop}, for every $p\in \mathrm{Per}(f)\cap \Lambda$
		with $\lambda^c(p)<0$, it holds
		$\Lambda= H(p,f)$. 
	\end{Lemma}
	
	\begin{proof}
		Fix $p\in \mathrm{Per}(f)\cap \Lambda$
		with $\lambda^c(p)<0$. 
		To prove the inclusion $H(p,f) \subset \Lambda$, 
		note that
		$\lambda^c(p)<0$ implies  $W^u (p) = \mathcal{F}^u (p)$. Thus it holds $H(p,f) \subset \overline{W^u(p)} \subset \Lambda$.
		
		For the inclusion $\Lambda \subset H(p,f)$, fix any $x\in \Lambda$ and  consider any small disk $\Delta$ of dimension
		$\dim (E^u)$ with
		$\Delta \subset \mathcal{F}^u(x)$ containing $x$ in its interior.
		We claim that $\Delta \pitchfork W^s(p)$. For simplicity of notation, let us assume that $p$ is a fixed point of $f$.
		Observe that the $\mathcal{F}^u$-minimality of $\Lambda$ implies that every sufficiently large $u$-disk transversely intersects 
		$W^s(p)$. By the uniform expansion of $f$ along unstable leaves, there is $k>0$ such $f^k(\Delta)$ is large
		enough and
		hence $f^k(\Delta) \pitchfork W^s(p)$
		and therefore $W^s(p) \pitchfork \Delta$. By the $\mathcal{F}^u$-minimality, $W^u(p)$ accumulates on $\Delta$. The partially hyperbolic structure implies
		that there  exists $y\in W^s(p) \pitchfork W^u(p)$ close to $\Delta$ and hence close to $x$. As this can be done with arbitrarily
		small disks $\Delta$, it follows that $x$ is accumulated by transverse homoclinic points of $p$ and  hence 
		belongs to $H(p)$.
	\end{proof}
	
	After these auxiliary results, we are ready to prove Proposition~\ref{main.prop}.
	
	\subsubsection{Proof of Proposition~\ref{main.prop}}
	\label{sss.proofmainprop}

	The periodic point $p$ in the proposition is obtained using Theorem~\ref{thm.Liao-Gan-shadowing}.
	Note that by Remark~\ref{r.liaoganapplication} we only need to check conditions (2) and (3) of that theorem.
	We now go to the details.

	Partial hyperbolicity of $\Lambda$ implies that there is a small neighborhood $U$ of $\Lambda$ such that the maximal invariant $K$ of $f$ in $U$ is partially hyperbolic with one dimensional center.
	Consider the continuous map
	\begin{equation}
		\label{e.defvarfi}
		\varphi \colon K\to\mathbb{R}, \quad \varphi(x)  \eqdef \log\|Df|_{E^c(x)}\|
	\end{equation}
	and  let 
	$$
	\|\varphi\|_{0}=\sup_{x\in K}|\varphi(x)|.
	$$ 
	By the uniform continuity of $\varphi$, given any $\delta>0$ there exists $\xi>0$
	\begin{equation} 
		\label{e.delta}
		|\varphi(x)-\varphi(y)|<\delta \quad \mbox{for every $x,y\in K$ with $\ud(x,y)<\xi$}.
	\end{equation}
	In the following, we will shrink $\delta$.  
	
	Given any invariant measure $\mu$ supported on $K$ consider its (averaged) center Lyapunov exponent 
	$\lambda^c(\mu)$ as in \eqref{e.averageexp}.
	As $\Lambda$ is mostly expanding, Proposition~\ref{p.bound-of-LE} implies that
	\begin{equation}\label{eq:choice-of-tau0}
		\tau_0  \eqdef \frac{1}{2}\cdot\inf\big\{\lambda^c(\mu) \colon \textrm{$\mu \in \mathrm{Gib}^u(\Lambda)$
		} \big\}>0.
	\end{equation}
	
	Let $q\in\Lambda \cap \mathrm{Per}(f)$  with $\lambda^c(q)<0$.  
	Note that
	\begin{equation}
		\label{e.laphi}
		|\lambda^c(q)| \le ||\varphi||_0.
	\end{equation}
	For simplicity, let us assume that $q$ is a fixed point of $f$.   Take $\e>0$ small such that  $B_\e(\Lambda)\subset U.$
	By shrinking $\delta$, we can assume that $4 \delta< \min\{ |\lambda^c(q)|, \tau_0\}$.
	
	Let $L>1$ and $d_0>0$ be the numbers given by  Theorem~\ref{thm.Liao-Gan-shadowing}   for $\frac{\lambda^c(q)}{2}$ and the dominated splitting $T_KM=(E^s\oplus E^c)\oplus  E^u$. Take 
	\begin{equation}\label{eq:choice-of-d}
		0<d<\min\big\{d_0, \xi/(L+1),\e/(L+1)\big\}.
	\end{equation}
	
	By the minimality of the lamination $\cF^{u}$  in $\Lambda$ and the uniform expansion of $f$ along the foliation $\cF^{u}$, there is an integer $N_d\in\NN$ such that $f^{N_d}(\cF^{u}_{d/4}(z))$ is $d/2$-dense in $\Lambda$ for every $z\in \Lambda$, where 
	$\cF^{u}_{\epsilon}(y)$ denotes the ball of radius $\epsilon$ centered at $y$ in $\cF^u(y)$.
	
	By Theorem~\ref{thm.empirical}, there is a point $x\in \cF^u_{d/4}(q)$ such that every weak$*$-limit of the  empirical measures $(\epsilon_n(x))$ defined in \eqref{e.empirical}
	are u-Gibbs states.  
	\[
	\liminf_{n\rightarrow\infty}\frac{1}{n}\sum_{i=0}^{n-1}\varphi(f^i(x))\geq2\tau_0.
	\] 
	By the choice of $\delta$, there is  $N_0\in\mathbb{N}$ such that
	\begin{equation}\label{eq:lower-bound-by-tau0}
		\frac{1}{n}\sum_{i=0}^{n-1}\varphi(f^i(x))>\tau_0+\delta, \,\textrm{for any $n\geq N_0$ }.
	\end{equation}
	
	By the choice of $N_d$, for every $n \in \mathbb{N}$   there is 
	$$
	\mbox{$z_n\in \cF^{u}_{d/4}(f^n(x))$ such that $\ud(f^{N_d}(z_n), q)<d/2.$ }
	$$
	Given any $m\in\mathbb{N}$, let  
	$$
	w_{m,n} \eqdef f^{-n-m}(z_n).
	$$ 
	By the uniform contraction of $f^{-1}$ along the strong unstable foliation and
	since  $x\in \cF^u_{d/4}(q)$, one has  $w_{m,n}\in \cF^{u}_{d/2}(q)$. Thus, recalling the choice of $z_n$, 
	\begin{equation}\label{eq:periodic-pseudo}
		\ud\big(w_{m,n}, f^{m+n+N_d}(w_{m,n})\big)
		=  \ud\big(w_{m,n}, f^{N_d}(z_{n})\big)
		<d.
	\end{equation}
	
	\begin{Claim-numbered}\label{c.time-controll}
		Given any $\delta\in\big(0,\frac{1}{4}  \min\big\{ |\lambda^c(q)|, \tau_0\big\}\big)$ small,	there exist $m,n\in\NN$ arbitrarily large and  a periodic point $p=
		p_{m,n}\in K$ such that  
		\begin{itemize}
			\item  $p$ has period $m+n+N_d$;
			\item  $\cO_p$ is $\left(\e,1+\dfrac{\lambda^c(q)}{2\|\varphi\|_{0}}\right)$-good for $\cO_q$;
			\item $\lambda^c(q) \cdot \dfrac{2\|\varphi\|_{0}-\tau_0 }{2\|\varphi\|_{0}}<\lambda^c(p)<\dfrac{\lambda^c(q)}{2}.$
		\end{itemize}
	\end{Claim-numbered}
	\proof We will firstly require that $n>N_0.$
	We now estimate the Birkhoff averages of $\varphi$ along the orbit of $w_{m,n}$.
	First, by the uniform expansion of $f$ along the foliation $\cF^{u}$, we can split the orbit segment 
	$\{w_{m,n}, f(w_{m,n}), \dots, f^{m+n-1} (w_{m,n})\}$ as follows
	\begin{equation}
		\label{e.wmnq}
		\begin{split}
			f^{i}(w_{m,n})&= f^{-n-m+i}(z_n) \in\cF^{u}_{d/4}(q) \quad \mbox{for every $0\leq i\leq m-1$};\\
			f^{i}(w_{m,n})&= f^{-n} (f^{i-m}(z_n))
			\in\cF^{u}_{d/4}(f^{i-m}(x)) \quad 
			\mbox{for every $m\leq i\leq m+n-1.$}
		\end{split}
	\end{equation}
	By the   choices of $d$ in Equation~\eqref{eq:choice-of-d} and $\xi$ in Equation~\eqref{e.delta}, for $0\leq k\leq m$ one has that 
	\begin{equation}\label{eq:LE-first-m-iterates}
		\left|\frac{1}{k}\sum_{i=0}^{k-1}\varphi(f^i(w_{m,n}))-\lambda^c(q)\right|<\delta.
	\end{equation}
	For  $m+1\leq k\leq m+n+N_d$, using \eqref{eq:LE-first-m-iterates}, one has that 
	\begin{equation}\label{eq:LE-till-end-upper}
		\sum_{i=0}^{k-1}\varphi(f^i(w_{m,n}))=\sum_{i=0}^{m-1}\varphi(f^i(w_{m,n}))+\sum_{i=m}^{k-1}\varphi(f^i(w_{m,n}))\leq m\cdot (\lambda^c(q)+\delta)+(k-m)\|\varphi\|_{0}.
	\end{equation}
	Moreover, by Equations~\eqref{eq:LE-first-m-iterates} and~\eqref{eq:lower-bound-by-tau0}
	(recall that $n> N_0$),  and the choice of $d$, one has 
	\begin{equation}\label{eq:LE-till-end-lower}
		\begin{split}
			\sum_{i=0}^{m+n+N_d-1}\varphi(f^i(w_{m,n}))
			&=\sum_{i=0}^{m-1}\varphi(f^i(w_{m,n}))+\sum_{i=m}^{m+n-1}\varphi(f^i(w_{m,n}))+\sum_{i=m+n}^{m+n+N_d-1}\varphi(f^i(w_{m,n}))
			\\
			&\geq m\cdot (\lambda^c(q)-\delta)+n\cdot \tau_0-N_d\cdot\|\varphi\|_{0}.
		\end{split}
	\end{equation}
	By Equation~\eqref{eq:LE-till-end-upper},  the following estimate holds for 
	every $m+1\leq k\leq m+n+N_d$,
	\begin{equation}\label{eq:LE-upper-bound-raw}
		\begin{split}
			\frac{1}{k}\sum_{i=0}^{k-1}\varphi(f^i(w_{m,n}))&\leq \frac{m}{k}\cdot (\lambda^c(q)+\delta)+\frac{(k-m)}{k}\|\varphi\|_{0}
			\\
			&=\|\varphi\|_{0}+\frac{m}{k}\cdot (\lambda^c(q)+\delta-\|\varphi\|_{0})
			\\
			&\leq \|\varphi\|_{0}+\frac{m}{m+n+N_d}\cdot (\lambda^c(q)+\delta-\|\varphi\|_{0})
			\\
			&= \frac{m}{m+n+N_d}\cdot (\lambda^c(q)+\delta)+\frac{n+N_d}{m+n+N_d}\|\varphi\|_{0}.
		\end{split}
	\end{equation}
	Thus for $m$ large enough (much larger than $N_d$), one has  
	\begin{equation}\label{eq:LE-pre-bound}
		\begin{split}
			\frac{1}{k}\sum_{i=0}^{k-1}\varphi(f^i(w_{m,n}))
			&\leq \frac{m}{m+n+N_d}\cdot (\lambda^c(q)+\delta)+\frac{n+N_d}{m+n+N_d}\|\varphi\|_{0}
			\\
			&\leq \frac{m}{m+n}\cdot \lambda^c(q) +\frac{n}{m+n}\|\varphi\|_{0}+\delta.
		\end{split}
	\end{equation}
	By taking $\delta>0$ small and $m,n$ large enough such that 
	\begin{equation}\label{eq:choice-of-mn}
		\frac{m}{n}\in\bigg(\frac{2\|\varphi\|_{0}+2\delta-\lambda^c(q)}{-\lambda^c(q)-2\delta},  \frac{2\|\varphi\|_{0}+3\delta-\lambda^c(q)}{-\lambda^c(q)-3\delta}\bigg),
	\end{equation}
	for every $m+1\leq k\leq m+n+N_d$, one has that
	\begin{equation}\label{eq:LE-final-bound}
		\begin{split}
			\frac{1}{k}\sum_{i=0}^{k-1}\varphi(f^i(w_{m,n}))  
			&\leq\frac{m/n}{m/n+1}\cdot \lambda^c(q) +\frac{1}{m/n+1}\|\varphi\|_{C^0}+\delta
			\\
			&=\frac{m/n}{m/n+1}\cdot \big(\lambda^c(q)-\|\varphi\|_{C^0}\big) +\|\varphi\|_{C^0}+\delta
			\\   
			&\leq  \frac{2\|\varphi\|_{C^0}+2\delta-\lambda^c(q)}{2\|\varphi\|_{C^0}+2\delta-\lambda^c(q)-\lambda^c(q)-2\delta}\cdot \big(\lambda^c(q)-\|\varphi\|_{C^0}\big) +\|\varphi\|_{C^0}+\delta
			\\
			&= \frac{\lambda^c(q)}{2}.
		\end{split}
	\end{equation}
	As $\delta<\frac{\lambda^c(q)}{4}$, combining with  Equations~\eqref{eq:LE-first-m-iterates} and \eqref{eq:LE-final-bound},  one gets that 
	\[ 	\frac{1}{k}\sum_{i=0}^{k-1}\varphi(f^i(w_{m,n}))  
	\leq\frac{\lambda^c(q)}{2}\quad\textrm{
		for every $1\le k \le m+n+N_d$}.\]  
	Combined with Equation \eqref{eq:periodic-pseudo}, by Theorem~\ref{thm.Liao-Gan-shadowing} and Remark~\ref{r.liaoganapplication}, there exists a periodic point $p$ of period $(m+n+N_d)$ such that 
	$$
	\ud(f^i(p),f^i(w_{m,n}))<L\cdot d, \textrm{ for $i=0,\cdots,m+n+N_d-1$}.
	$$
	Therefore using \eqref{e.wmnq} one has 
	\begin{itemize}
		\item	for $0\leq i\leq m-1\colon \ud(f^i(p),f^i(q))\leq \ud(f^i(p),f^i(w_{m,n}))+\ud(f^i(w_{m,n}),f^i(q))<(L+1)d;$
		\item  for  $m\leq i\leq m+n-1\colon \ud(f^i(p),f^i(x))\leq \ud(f^i(p),f^i(w_{m,n}))+\ud(f^i(w_{m,n}),f^i(x))<(L+1)d.$
	\end{itemize}
	
	We now estimate the center Lyapunov exponent of $p$.  Since  $(L+1)d\leq\xi$, recall  \eqref{eq:choice-of-d}, and $m$ can be chosen arbitrarily large, by the choice of $\xi$, following the lines of  the same estimate as in Equations~\eqref{eq:LE-first-m-iterates}, ~\eqref{eq:LE-till-end-upper}, \eqref{eq:LE-upper-bound-raw}, \eqref{eq:LE-pre-bound} and \eqref{eq:LE-final-bound}, one has  the upper bound  of the center Lyapunov exponent of $p$, 
	\begin{align*}
		\lambda^c(p)&=\frac{1}{m+n+N_d}\sum_{i=0}^{m+n+N_d-1}\varphi(f^i(p))
		\\
		&\leq \frac{m}{m+n+N_d}\cdot (\lambda^c(q)+\delta)+\frac{n+N_d}{m+n+N_d}\|\varphi\|_{0}
		\\
		&\leq\frac{m/n}{m/n+1}\cdot \lambda^c(q) +\frac{1}{m/n+1}\|\varphi\|_{0}+\delta
		\\
		&\leq \frac{\lambda^c(q)}{2}.
	\end{align*}
	A lower bound of the center Lyapunov exponent of $p$ is obtained following the estimate in Equation~\eqref{eq:LE-till-end-lower},
	\begin{align*}
		\lambda^c(p)=\frac{1}{m+n+N_d}\sum_{i=0}^{m+n+N_d-1}\varphi(f^i(p))&\geq 	\frac{1}{m+n+N_d}\big(m(\lambda^c(q)-\delta)+n\tau_0-N_d\|\varphi\|_{0}\big)
		\\
		&	\geq \frac{m}{m+n}  \lambda^c(q) +\frac{n}{m+n}\tau_0-2\delta.
	\end{align*}
	Since, by  \eqref{eq:choice-of-mn},
	\[
	\frac{m}{n}< \frac{2\|\varphi\|_{0}+ 3\delta-\lambda^c(q)}{-\lambda^c(q)-3\delta},
	\]
	by choosing $\delta>0$ small,  one has 
	\begin{align*}
		\lambda^c(p)&\geq \frac{m}{m+n}  \lambda^c(q) +\frac{n}{m+n}\tau_0-2\delta
		\\
		&=\lambda^c(q) +\frac{n}{m+n}\big(\tau_0-\lambda^c(q)\big)-2\delta
		\\
		&\geq \lambda^c(q)+\frac{\big(\tau_0-\lambda^c(q)\big)\big(-\lambda^c(q)-3\delta\big)}{2\|\varphi\|_{C^0}-2\lambda^c(q)}-2\delta
		\\
		&\geq \lambda^c(q)+\frac{\tau_0\cdot (-\lambda^c(q))}{2\|\varphi\|_{0}}
		\\
		&= \lambda^c(q) \cdot \frac{2\|\varphi\|_{0}-\tau_0 }{2\|\varphi\|_{0}}.
	\end{align*} 
	This completes the estimates of the Lyapunov exponent of $p$ in the claim,

	We now estimate the proportion of time that the orbit segment stays close to the orbit of $q$. By Equation~\eqref{eq:choice-of-mn}, for $n$ large enough (much larger than $N_d$), one has 
	\begin{align*}
		\frac{m}{m+n+N_d}=	\frac{m/n}{m/n+1+N_d/n}&> \frac{2\|\varphi\|_{0}+ 2\delta-\lambda^c(q)}{2\|\varphi\|_{0}-2\lambda^c(q)+N_d(-\lambda^c(q)-2\delta)/n}
		\\
		&\geq \frac{2\|\varphi\|_{0}-\lambda^c(q)}{2\|\varphi\|_{0}-2\lambda^c(q)}
		\\
		&\geq 1+\frac{\lambda^c(q)}{2\|\varphi\|_{0}}.
	\end{align*}
	Since $(L+1)d<\e$, the periodic orbit $\cO_p$ is $\big(\e,1+\frac{\lambda^c(q)}{2\|\varphi\|_{0}}\big)$-good for $\cO_q$,
	proving the claim.
	\endproof 
	By  Claim~\ref{c.time-controll}, it suffices to take \[\rho=\frac{1}{2\|\varphi\|_{0}}\quad \textrm{and}\quad \eta=\frac{2\|\varphi\|_{0}-\tau_0 }{2\|\varphi\|_{0}}.\] 
	
	Now, it remains to show that $\cO_p$ is contained in $\Lambda$. To do this, we need to shrink $\e$. As $\cO_p$ is in the $\e$-neighborhood of $\Lambda$ it  is contained in the maximal invariant set  $K$ of $f$ in $U$ fixed in the beginning of the proof. 
	Recall that by Remark~\ref{r.liaogan}
	the stable manifold of $p$   has uniform size 
	(depending only on $\lambda^c(q)/2$). 
	As $p$ is $\e$-close to $q$ and their invariant manifolds
	have the same dimension, if $\e$ is small enough, $p$ and $q$ are homoclinically related,
	recall Remark~\ref{r.intersectionstun}. By the Inclination lemma,  one has 
	\[
	W^u(\cO_p)\subset \overline{W^u(\cO_q)}=\overline{\cF^u(\cO_q)} \subset \Lambda,
	\]
	where in the last inclusion  we use that
	$\Lambda$  is $\mathcal{F}^u$-minimal. Thus $\cO_p\subset\Lambda$, ending the proof of the proposition.
	\hfill \qed

	Now, we are ready to give the proof of Theorem~\ref{thm.existence}.
	
	\subsection{End of the Proof of Theorem~\ref{thm.existence}}
	\label{ss.endproofE}
	Recall that, by Lemma~\ref{l.homorelated}, every
	pair of periodic orbits in $\Lambda$ with negative center Lyapunov exponent are homoclinically related. 
	
	Let $\cO_p$ be a periodic orbit with negative center Lyapunov exponent. Let $\rho>0$ and $\eta\in(0,1)$ be the constants given by Proposition~\ref{main.prop}. Let $\e_n= 2^{-n}$. We will inductively construct a sequence of periodic orbits $\{\cO_{p_n}\}$ in $\Lambda$ such that
	\begin{itemize}
		\item $\cO_{p_{n+1}}$ is $\big(\e_{n},1-2\rho|\lambda^c(p_n)|\big)$-good for the periodic orbit $\cO_{p_n}$;
		\item  $\cO_{p_{n+1}}$ is $\e_n$-dense in $\Lambda.$
		\item  $\frac{1+\eta}{2}\cdot \lambda^c(p_{n})<\lambda^c(p_{n+1})<0.$
	\end{itemize} 
	Let $p_1=p$. Assume already constructed the periodic orbits $\cO_{p_1},\cdots,\cO_{p_n}$ satisfying the above properties. Applying  Proposition~\ref{main.prop} to  $\cO_{p_n}$ with $\e=\e_n/2$, 
	one gets a periodic orbit $\tilde\cO_{p_{n+1}}\subset\Lambda$ such that 
	\begin{itemize}
		\item $\tilde\cO_{p_{n+1}}$ is $\big(\e_{n}/2,1-\rho|\lambda^c(p_n)|\big)$-good for  $\cO_{p_n}$
		and
		\item  $\eta\cdot \lambda^c( p_{n})<\lambda^c(\tilde p_{n+1})<0.$
	\end{itemize} 
	
	By Lemma~\ref{l.homoclass}, the homoclinic class of $\tilde\cO_{p_{n+1}}$  coincides with $\Lambda$. By Remark~\ref{r.good-approximation-in-homoclinic-class},  there  exists a periodic orbits $\cO_{p_{n+1}}$ which is homoclinically related with  $\tilde\cO_{p_{n+1}}$ such that 
	\begin{itemize}
		\item $\cO_{p_{n+1}}$ is $\big(\e_{n}/2,1-\rho|\lambda^c(p_n)|\big)$-good for $\tilde\cO_{p_{n+1}}$;
		\item $\cO_{p_{n+1}}$ is $\e_n$-dense in $\Lambda$; and
		\item  $\lambda^c(\tilde p_{n+1})+(1-\eta)\lambda^c(p_n)/2<\lambda^c(p_{n+1})<0.$
	\end{itemize}
	This implies that 
	\begin{itemize}
		\item $\cO_{p_{n+1}}$ is $\big(\e_{n},1-2\rho|\lambda^c(p_n)|\big)$-good for  $\cO_{p_n}$;
		\item    $\cO_{p_{n+1}}$ is $\e_n$-dense in $\Lambda$; and
		\item  $\frac{1+\eta}{2}\lambda^c(p_{n})<\lambda^c(p_{n+1})<0.$
	\end{itemize} 
	By construction, one has that 
	$$
	|\lambda^c(p_{n})|<\left( \frac{1+\eta}{2} \right)^n\cdot|\lambda^c(p)|.
	$$ 
	Taking 
	\[
	\kappa_n \eqdef 1-2\rho|\lambda^c(p_n)|>1-2\rho\left(\frac{1+\eta}{2}\right)^n\cdot|\lambda^c(p)|
	\] 
	we get that the orbit $\mathcal{O}_{p_{n+1}}$ is $(\varepsilon_n, \kappa_n)$-good for  $\mathcal{O}_{p_{n}}$.
	
	It remains to prove the properties of the sequences $(\varepsilon_n)$ and $(\kappa_n)$
	Note that $\sum\e_n<\infty$ and $\prod\kappa_n\in(0,1]$.  By Theorem~\ref{thm.GIKN-criterion},  the sequence
	of periodic measures $\delta_{\cO_{p_n}}$ converges to an ergodic measure $\nu$ whose support is $\Lambda$.

	Note that $\lambda^c ( \delta_{\cO_{p_n}}) \to 0$. As $\dim(E^c)=1$, 
	By Remark~\ref{r.convergenceofexp}
	$\lambda^c(\nu)$ is zero and hence $\nu$ is nonhyperbolic.
	The proof of the theorem is now complete. 
	\hfill  \qed
	
	\begin{Remark}{\rm{
				Using Remark~\ref{r.C1-u-Gibbs}, one can replace the mostly expanding property on $\Lambda$ by   the property that 
				{\em{every measure satisfying the unstable entropy formula has positive center Lyapunov exponent}},  then  Proposition~\ref{p.intermediate} and Theorem~\ref{thm.existence} hold in $C^1$-setting.}}
	\end{Remark}

	\section{Hyperbolicity versus  nonhyperbolic measures}
	\label{s.hypversus}

	In this section, we prove Theorem~\ref{t.transversecondition}.
	For that we first state 
	the following dichotomy for 
	mostly expanding $\mathcal{F}^u$-minimal sets: 
	
	\begin{Proposition}\label{p.intermediate}
		Let $f\in\ph^{1+\alpha}_{c=1}(M)$ and  $\Lambda$ be an $f$-invariant and 
		$\mathcal{F}^u$-minimal set which is mostly expanding. 
		Then the followings dichotomy holds:
		\begin{itemize}
			\item either $\Lambda$ supports  some  nonhyperbolic  measure; 
			\item  or  $\Lambda$ is uniformly hyperbolic with uniformly expanding center. 
		\end{itemize}
	\end{Proposition}
	
	We postpone the proof of this proposition to Section~\ref{ss.pintermediate} and prove Theorem~\ref{t.transversecondition} assuming it
	in Section~\ref{ss.transversecondition}.

	\subsection{Proof of Theorem~\ref{t.transversecondition}}\label{ss.transversecondition}
	
	We begin by discussing a transversality condition introduced in  \cite[Definition 1.1]{AviCroWil:},  see also \cite{CroPot:}.
	A local version of it can be found in  \cite[Definition~5.18]{CroPot:15}. 
	
	Let us recall some terminologies from \cite[Sections 1.2 and 1.3]{AviCroWil:}.
	Let $f\in \mathrm{PH}^1_{c=1}(M)$.
	Fix $\varepsilon_0>0$ small. A \emph{brush} through a point $x\in M$ is given by 
	\[{\rm Br}(x)=\bigcup_{y\in\cF^u_{\varepsilon_0}(x)}\cF^s_{\varepsilon_0}(y)\]
	which is a co-dimension one topological submanifold.  
	It has been proven in  \cite[Lemma 4.4]{AviCroWil:} that for every $\varepsilon\ll\varepsilon_0$, the set 
	\[{\rm U}_{\varepsilon}(x)=B(x,\varepsilon)\setminus{\rm Br}(x) \]
	has exactly two connected components, where $B(x,{\varepsilon})$ is the ball centered at $x$ of radius $\varepsilon$.    This allows to locally define the two sides of $x$. Two points $y_1,y_2\in{\rm U}_{\varepsilon}(x)$ \emph{lie on the same (or different) side(s) of ${\rm Br}(x)$} if they belong to  the same (or different) connected component(s) of ${\rm U}_{\varepsilon}(x)$. Note that  the brushes can be extended canonically along any strong unstable leaf. It is proved in \cite[Lemma 4.5]{AviCroWil:} that the notion of lying on the same side of ${\rm Br}(x)$ can be locally extended along any path $\gamma$ in $\cF^u(x)$, and this allows to define  $y_0\in {\rm U}_{\varepsilon}(\gamma(0))$ and $y_1\in {\rm U}_{\varepsilon}(\gamma(1))$ lying on the same (or different) side(s) of ${\rm Br}(\gamma)$. This gives rise to the notion of $y_0,y_1$ \emph{lying on the same (or different) side(s) of $\cF^u(x)$ relative to $\gamma$}. 
	See \cite[Section 4.2]{AviCroWil:} for more details. 
	
	\begin{Definition}[$S$-transversality]\label{d.transversecondition}
		Let $f\in \mathrm{PH}^1_{c=1}(M)$ be and $\Lambda$ be a $\cF^u$-saturated set of $f$. 
		The set $\Lambda$ is \emph{$s$-transverse}, 
		if for any $\tau>0$ small enough and any point $x\in\Lambda$, there are paths $\gamma_1,\gamma_2:[0,1]\to\cF^u(x)$ such that 
		\begin{itemize}
			\item
			$d(\gamma_1(t),\gamma_2(t))< \tau$ for any $t\in[0,1];$ 
			\item $\gamma_2(0)\in {\rm U}_{\tau}(\gamma_1(0))$ and $\gamma_2(1)\in{\rm U}_\tau(\gamma_1(1))$, but $\gamma_2(0)$ and $\gamma_2(1)$ lie on different sides of $\cF^u(x)$ relative to $\gamma_1$.	 
		\end{itemize}
	\end{Definition}

	Checking the transversality condition is in general a difficult task.
	Next result illustrates its relevance.
	

\begin{Theorem}[Theorem C in \cite{AviCroWil:}] \label{r.transversality}
	{\em{Let $f\in \mathrm{PH}^{1+\alpha}_{c=1}(M)$ and $\Lambda$ be $\cF^u$-saturated and $f$-invariant set. 
			If $\Lambda$ is $s$-transverse and satisfies the SH condition\footnote{This condition was introduced in \cite{PujSam:06} as an ingredient to get minimality of strong foliations replacing the use of blender horseshoes in \cite{BonDiaUre:02}.} (the latter being verified if the center is uniformly expanding), then 	$\Lambda$ contains a $cu$-disk.}} 
\end{Theorem}

To prove the theorem we start by recalling an important consequence of accessibility:

\begin{Theorem}[\cite{BurWil:10,RodRodUre:08}]
	\label{t.thenitisergodic}
	Let   $f\in\ph_{\mathrm{Leb},c=1}^{1+\alpha}(M) $ be accessible. Then $\mathrm{Leb}$ is ergodic.
\end{Theorem}

\proof[End of the proof of Theorem~\ref{t.transversecondition}] Note that, by Theorem~\ref{thm.mostly-expanding-criterion}, accessibility of $f$ and  $\lambda^c(\mathrm{Leb})>0$
imply that $f$ is mostly expanding. Hence the hypotheses of Proposition~\ref{p.intermediate} are satisfied. Thus to prove
the theorem it remains to check that if $\Lambda$ is uniformly hyperbolic with an uniformly expanding center then it is Anosov.
Consider the $cu$-disk $\Delta \subset \Lambda$ provided by Remark~\ref{r.transversality}. It is enough to prove that the
orbit of $\Delta$ is dense in the ambient space and hence $\Lambda=M$.

Take a point $z$ in the  interior of $\Delta$ and observe that   there exist $\varepsilon>0$ and  $\delta>0$ such that
if $d(x,z)< \delta$ then $\mathcal{F}^s_\varepsilon (x)$ transversely intersects $\Delta$. 
Since $f$ is accessible, by Theorem~\ref{t.thenitisergodic} it is ergodic and hence transitive. Take $y$ with a forward dense orbit. 
Note that every point in the strong stable manifold of $y$ also has a forward dense orbit.
Note that there is $n\in\mathbb{N}$ such that $d(f^n(y),z)< \delta$ and hence $\mathcal{F}^s_\varepsilon (f^n (y)) \cap \Delta) \ne \emptyset$.
This implies that the forward orbit of $\Delta$ is also dense, proving the assertion.
\endproof

\subsection{Proof of Proposition~\ref{p.intermediate}}\label{ss.pintermediate}
We start with a preliminary result:

\begin{Lemma}\label{l.signalofexponents}
	If the set $\Lambda$ supports an ergodic measure $\nu$ with $\lambda^c(\nu)<0$,  then it supports a nonhyperbolic measure.
\end{Lemma}

\begin{proof}
	It $\Lambda$ supports a  hyperbolic measure $\nu$ with $\lambda^c(\nu)<0$, by \cite{Cro:11,Gel:16}, there are  hyperbolic periodic points $p$  (arbitrarily close to $\Lambda$) with $\lambda^c(p)<0$. Moreover,  every $\nu$-typical point $x\in \Lambda$ 
	and  the point $p$ have stable manifolds tangent to $E^{cs}$ with uniform size (depending only on 
	$\lambda^c(\nu)$).
	
	Using Remark~\ref{r.intersectionstun}, given $p$ and any $\nu$-generic point $x$,
	the minimality of $\mathcal{F}^u$ in $\Lambda$ and the fact that sizes of $W^s(p)$ and $W^s(x)$
	are uniform  imply that if  $p$ is close enough to $\Lambda$ then
	$W^s(p) \pitchfork \mathcal{F}^u(x)\ne \emptyset$. A similar argument gives that
	$W^s(x)$ has non-empty transverse intersection with $\mathcal{F}^u(p)=W^u(p)$.
	By the inclination lemma, one has 
	\[
	\mathcal{F}^u(p)\subset \overline{\bigcup_{i\in\mathbb{Z}}\mathcal{F}^u(f^i(x))} \subset \Lambda.
	\] 
	Thus one gets that $p\in\Lambda$. By Theorem~\ref{thm.existence}, the set $\Lambda$ supports non-hyperbolic measures, proving the lemma.
\end{proof}

We are now ready to prove the proposition. If $\Lambda$ supports a non-hyperbolic measure, we fall into the first case of the dichotomy. Thus, we may assume that all ergodic measures supported on $\Lambda$ are hyperbolic.
In this case, by Lemma~\ref{l.signalofexponents}, all ergodic measures supported on $\Lambda$ have a positive center Lyapunov exponent. Note that, under this assumption, every invariant measure $\mu$ must also have a positive exponent. This follows from the fact that its ergodic decomposition consists solely of measures $\nu$ with $\chi^c(\nu) > 0$.
Furthermore, there cannot exist a sequence of ergodic measures $(\mu_n)$ such that $\chi^c(\mu_n) \to 0$. To see this, considering a subsequence if necessary,  we can assume that $\mu_n \to \mu_\infty$. By Remark 2.1, $\chi^c(\mu_\infty) = 0$, which contradicts the earlier observation.

We claim that if every ergodic measure $\mu$ supported on $\Lambda$ satisfies $\chi^c (\nu)>0$ then the set  $\Lambda$ is hyperbolic. 
We follow the classical arguments in \cite[page 521]{Man:82}.
Otherwise, as $E^s$ and $E^u$ are uniformly contracting and uniformly expanding directions, respectively, there are sequences of points
$(x_k)$, of natural numbers $(n_k)$, $n_k\to \infty$, and  of  nonnegative numbers $(\varepsilon_k)$, $\varepsilon_k\to 0$, such that
(recalling the definition of $\varphi$ in \eqref{e.defvarfi})
$$
\frac{1}{n_k} \sum_{i=0}^{n_k-1} \varphi (f^i (x_k)) \in  (-\varepsilon_k, \varepsilon_k).
$$
Considering the sequence of  measures
$$
\nu_k \eqdef
\frac{1}{n_k} \sum_{i=0}^{n_k-1} \delta_{f^i (x_k)},
$$
the Krylov-Bogolyubov procedure provided a subsequence of $(\nu_k)$ converging to an invariant measure $\nu_\infty$. By construction, $\chi(\nu_\infty)=0$, which is a contradiction. This ends the proof of the proposition.
\qed.

\section{Diffeomorphisms homotopic to Anosov: Proof of Theorem~\ref{thm.DA-case}} \label{s.dicDA}

In this section, we prove Theorem~\ref{thm.DA-case}. To do so, in Section~\ref{ss.homotopic}, we
collect some results about diffeomorphisms homotopic to Anosov involving notions as coherence, accessibility, ergodicity, and
hyperbolicity. 
The proof of Theorem~\ref{thm.DA-case}
is given in Section \ref{ss.proofda}.

\subsection{Partially hyperbolic diffeomorphisms homotopic to Anosov on $\mathbb{T}^3$}
\label{ss.homotopic}
We collect some preliminary general results.
It is well known \cite{Fra:70} that for every diffeomorphism $f\in\diff^1(\mathbb{T}^3)$ that is 
homotopic to a linear diffeomorphism Anosov $A$, there is a continuous surjective map  $h\colon \TT^3\to\TT^3$ such that 
$h\circ f=A\circ h$. In other words, the map $f$ is semi-conjugate to $A$ via the map $h$. 
We call $A$ the {\em{linear part of $f$.}}
Furthermore, if $f\in \ph^1_{c=1} (\TT^3)$ then the center bundle is integrable. 
We consider first three results that do not require the conservative hypotheses.

\begin{Theorem}[Theorem A in \cite{Pot:15} and Remark 4.4 in \cite{Pot:14}]\label{th.dynamically-coherent-DA}
	Let $f\in\ph^1_{c=1}(\TT^3)$ be  homotopic to a linear Anosov diffeomorphism $A$, then 
	\begin{itemize}
		\item $A$ has simple spectrum consisting of real eigenvalues; 
		\item  $f$ is dynamically coherent;
		\item $f$ has a unique $\mathcal{F}^s$-minimal invariant set and a unique $\mathcal{F}^u$-minimal  invariant set.
	\end{itemize}
\end{Theorem}

In what is  below, we consider simultaneously a diffeomorphism $f$ and its Anosov factor $A$, to avoid misunderstanding
we emphasize the dependence on the maps of the splittings and the foliations,
denoting by $\mathcal{F}^*$ and $\cW^*$ the invariant foliations of $f$ and $A$ tangent to the bundles $E^*_f$ and $E^*_A$,
$*=s,c,u, cs, cu$.
In the next result we  summarize the properties of  the semi-conjugacy given by \cite{Fra:70}.

\begin{Theorem}[\cite{Ham:13,HamPot:14,Pot:15,Ure:12,BuFiSaVa:12,ViYan:17}]\label{thm.semi-conjugacy-preservinsing-center-foliation}
	Let $f\in\ph^1_{c=1}(\TT^3)$ be  homotopic to a linear Anosov diffeomorphism $A$ with semi-conjugacy $h$. 	
	Then 
	\begin{enumerate}
		\item  $h(\cF^c(x))=\cW^c(h(x))$ for every $x\in\TT^3$;
		\item 
		for every $x\in\mathbb{T}^3$, the pre-image $h^{-1}(x)$ is a 
		connected set  contained in $\mathcal{F}^c(h^{-1}(x))$;
		\item 
		if $A$ has contracting center, then  $h(\cF^u(x))=\cW^u(h(x)),$ for every $x\in\TT^3$;
		\item 
		the diffeomorphism $f$ has a unique measure of maximal entropy, $\mu_{\mathrm{max}}$, which is the only measure satisfying $h_\ast (\mu_{\rm max}) = \mathrm{Leb}$. Moreover, the center Lyapunov exponents $\chi^c(\mu_{\mathrm{max}},f)$ and $\chi^c(\mathrm{Leb},A)$ have the same sign;	
		\item  each center leaf intersects each strong unstable leaf in at most one point.   
	\end{enumerate}
\end{Theorem}

\begin{Remark}\label{r.countable-preimages}
	{\em{Item (2) of Theorem~\ref{thm.semi-conjugacy-preservinsing-center-foliation}
			implies that each center leaf $\cW^c(x)$ contains countably many 
			(perhaps finitely many) points whose pre-images under $h$ are not singletons. Therefore
			there are points  $\tilde x_0\in\mathbb{T}^3$ such that $h^{-1}(\tilde x_0)$ is a singleton.}}
\end{Remark}

We conclude our summary of general results on diffeomorphisms homotopic to Anosov with one
involving the transversality condition.

\begin{Theorem}[\cite{AviCroEskPotWilZha:}]
	\label{t.acepwz}
	Let $f\in\ph_{c=1}^{1+\alpha}(\mathbb{T}^3)$  be homotopic to a  linear Anosov 
	diffeomorphism $A$ with expanding center bundle.   
	Let $h$ be the semiconjugacy between $f$ and $A$.
	Then either the strong unstable foliation $\mathcal{F}^u$ is mapped by  $h$ to the strong unstable foliation of A,
	or $\mathcal{F}^u$ is $s$-transverse. In particular, if $f$ is accessible, then $\mathcal{F}^u$ is $s$-transverse.
\end{Theorem}

We now focus on the conservative setting, where sharper results can be achieved. We explore the relationships between accessibility, ergodicity, and hyperbolicity. The first result in this direction was presented in Theorem~\ref{t.thenitisergodic}. We now state a second key result.

\begin{Theorem}[\cite{GanShi:20}]\label{thm.ergodicity-of-DA}  
	Every $f\in\ph^{1+\alpha}_{\mathrm{Leb}, c=1}(\TT^3)$  homotopic to a linear Anosov diffeomorphism is ergodic.
	Moreover, if $f$ is 
	not accessible, then it is Anosov. 
\end{Theorem}
The next result comes from \cite{GanShi:20}.
\begin{Theorem}[Theorem 5.1 in \cite{GanShi:20}]\label{thm.gan-shi}
	Let $f\in\ph^{1+\alpha}_{ c=1}(\TT^3)$ be an Anosov diffeomorphism with expanding center and whose linear part is $A$.
	Let $h$ be the conjugacy between $f$ and $A$. If $h$ sends the strong unstable foliation of $f$ to the strong unstable foliation of $A$, then $f$ is not accessible.
\end{Theorem}

We now  prove our results for diffeomorphisms homotopic to Anosov.

\subsection{Proof of Theorem~\ref{thm.DA-case}}
\label{ss.proofda}
There are two cases based on the accessibility of $f$. If $f$ is not accessible, then by Theorem~\ref{thm.ergodicity-of-DA}, the map $f$ is Anosov, placing us in the first case of the dichotomy. Therefore, in what follows, we assume that $f$ is accessible.
Theorem~\ref{t.thenitisergodic} implies that the measure $\mathrm{Leb}$ is $f$-ergodic. If $\mathrm{Leb}$ is nonhyperbolic, we are in the second case of the dichotomy. Thus, it remains to consider the case where $\mathrm{Leb}$ is hyperbolic. Additionally, without loss of generality, we may assume (by considering $f^{-1}$ if necessary) that $\lambda^c(\mathrm{Leb}) > 0$.

As $f$ is accessible, by Theorem~\ref{thm.mostly-expanding-criterion}, $f$ is mostly expanding. By the third item in Theorem~\ref{th.dynamically-coherent-DA},   $f$ has a unique $\mathcal{F}^{u}$-minimal
invariant set $\Lambda$.
As $f$ is mostly expanding, every u-Gibbs state supported on $\Lambda$ has positive center Lyapunov exponent. 
By Proposition~\ref{p.intermediate}, the set $\Lambda$ either supports some nonhyperbolic measure or it is uniformly hyperbolic with uniformly expanding center. If the first case, we are in the first case of the dichotomy and the result follows. 
Thus it remains to address the latter case. For that we analyze the
central bundle of the linear part $A$ of $f$.

\begin{Claim-numbered}
	\label{cl.linearpartofA}
	The center bundle of $A$ is expanding.
\end{Claim-numbered}

Note that if the center bundle of $A$ is expanding, then the hypotheses of Theorem~\ref{t.acepwz}
hold. As the diffeomorphisms is accessible, 
the set $\Lambda$ satisfies the transversality condition. Consequently, the hypotheses of Theorem~\ref{t.transversecondition} are met, and the dichotomy in the theorem follows.

\begin{proof}[Proof of Claim~\ref{cl.linearpartofA}]
	Assume, by contradiction, that the central bundle is contracting. 
	Consider the semi-conjugacy $h$ between $f$ and $A$ provided by Theorem~\ref{thm.semi-conjugacy-preservinsing-center-foliation}. It
	maps the strong unstable leaves in $\Lambda$ into the strong unstable leaves of $A$ due Item (3) in Theorem~\ref{thm.semi-conjugacy-preservinsing-center-foliation}. This implies that  $h(\Lambda)=\mathbb{T}^3$.
	Hence, by Item (4) of Theorem~\ref{thm.semi-conjugacy-preservinsing-center-foliation},
	the unique  measure of maximal entropy of $f$ is supported on $\Lambda$ and has negative central Lyapunov exponent
	(the same sign as the one of $A$).
	This contradicts the fact that
	$\Lambda$ is central expanding, proving the claim.
\end{proof}

The proof of Theorem~\ref{thm.DA-case} is now complete.

\section{Time-one maps of geodesic flows: Proof of Theorem~\ref{thm.geodesic-flow-case}} 
\label{s.geodesicflow}

In this section we consider  partially hyperbolic diffeomorphisms $f$ on the whole ambience $M$. 
In Section~\ref{ss.thm.geodesic-flow-case}, we prove Theorem~\ref{thm.geodesic-flow-case}. 
To do  so, we first introduce some preliminary results in Section~\ref{ss.strongacc}.

\subsection{Preliminaries: strong forms of accessibility}
\label{ss.strongacc}
We begin by presenting a strong form of accessibility.
The diffeomorphism $f$ is \emph{usu-accessible} if there exists a constant
$L>1$ such that for every pair of points on $x,y\in M$ there are $C^1$-curves $\gamma_1,\gamma_2,\gamma_3$ of length at most $L$  whose concatenation connects $x$ and $y$ such that $\gamma_1,$ $\gamma_3$ are tangent to  $E^u$ and $\gamma_2$ is tangent to  $E^s.$ One can analogously define {\em{sus-accessibility.}}

%

Let us observe that  usu-accessibility implies that there is  a unique $\cF^u$-minimal set: 

\begin{Theorem}[Theorem 1.1 in \cite{RodUre:19}]\label{thm.unique-u-minimal}
	Let $f$ be a $usu$-accessible partially hyperbolic diffeomorphism.  Then the strong unstable foliation has a unique minimal set, and thus it is invariant. 
\end{Theorem}

We now consider accessibility in the specific context of geodesic flows on negatively curved surfaces. It is well known that every geodesic flow on a negatively curved surface is Anosov and has minimal strong stable and unstable foliations \cite{Ano:67, Pla:72}. Additionally, it is important to note that the strong bundle distribution
$E^s\oplus E^u$  of any geodesic flow on such a surface is not integrable \cite{Pla:72}. With these considerations in mind, the following result is a direct consequence of \cite{RodUre:19, Wilk:98}. 

\begin{Theorem}[Theorem 4.1 in  \cite{RodUre:19} and Lemma 2.1 in \cite{Wilk:98}]
	\label{thm.both-usu-accessible-of-geodesic-flow}
	Let $(\varphi_t)_{t\in\mathbb{R}}$ be the geodesic flow on a negatively curved surface.    Then there exists a $C^1$-open neighborhood of $\cU$ of $\varphi_1$ such that every $f\in\cU$ is both usu-accessible and sus-accessible.
\end{Theorem}

We will use the following auxiliary result from \cite{CroPol:23}, see also \cite{ViYan:13}. 
Given a partially hyperbolic diffeomorphism $f$ with center bundle $E^c$, for $\chi >0$ define the set
\begin{equation}
	\label{e.Bchi}
	B^c_{\chi}  \eqdef
	\left\{x \in M \colon \, \liminf_{n\rightarrow+\infty}\frac{1}{n}\log\|Df^n|_{E^c(x)}\|\geq \chi\right\}.
\end{equation}

As the center foliation of the time-one map of an Anosov flow is $C^1$, by \cite{HirPugShu:77},  each $C^1$-nearby diffeomorphism is dynamically coherent.  In the next proposition, $\cF^c_{f,r}(x)$ 
denotes the ball of radius $r$ centered at $x$ in $\cF^c_f(x)$.

\begin{Proposition}[Proposition 8.3 in \cite{CroPol:23}]\label{p.finite-pliss}
	Let $\phi_1$ be the time one-map of an Anosov flow. Then there exists a $C^1$-open neighborhood $\cU$ of $\phi_1$ such that for every $\chi>0$ and every $r>0$, there is $k_0\in\mathbb{N}$ such that for every $f\in\cU$ and 
	every $x\in M$, one has  
	\[
	\#\big (\cF^c_{f, r}(x)\cap B_\chi^c \big)
	\leq k_0.  \]
\end{Proposition}

\subsection{Proof of Theorem~\ref{thm.geodesic-flow-case}}\label{ss.thm.geodesic-flow-case}
By Theorem 7.1 in \cite{HirPugShu:77}, there exists a $C^1$-neighborhood $\cV$ of $\phi_1$
such that every $f\in\cV$ is dynamically coherent and every  leaf of the one dimensional center foliation 
of $f$ is  $f$-invariant. By Theorem~\ref{thm.both-usu-accessible-of-geodesic-flow}, up to shrinking $\cV$, one can assume that each  $f\in\cV$ is both usu-accessible and sus-accessible. In particular, by Theorem~\ref{thm.unique-u-minimal}, $f$ has a unique $\mathcal{F}^u$-minimal set $\Lambda$, 
thus $\Lambda$ is $f$-invariant. 

Let $\cU=\cV\cap\diff^1_{\mathrm{Leb}}(M)$. We can take a $C^2$-diffeomorphism $f\in \cU$.
As $f$ is usu-accessible, 
by Theorem~\ref{t.thenitisergodic}, the volume measure $\mathrm{Leb}$ is $f$-ergodic.
Up to considering $f^{-1}$ one  can assume that 
$$
\lambda^c (\mathrm{Leb})=\int\log\|Df|_{E^c}\|\ud\mathrm{Leb}\geq 0.
$$ 
If 
$\lambda^c (\mathrm{Leb})= 0$, one gets a nonhyperbolic measure, proving the theorem. Therefore, we 
can assume
that $\lambda^c (\mathrm{Leb})>0$ and hence,  by Theorem~\ref{thm.mostly-expanding-criterion}, $f$ is mostly expanding.
Thus,  by Proposition~\ref{p.intermediate}, either $f$ has some nonhyperbolic measure on $\Lambda$ (and we are done) or the  $\mathcal{F}^u$-minimal set $\Lambda$ is uniformly hyperbolic  with expanding center.  
To complete the proof of the theorem it is enough to prove the claim below. For that consider  
the following constants
\begin{equation}
	\label{e.Dbounds}
	\begin{split}
		C  \eqdef &\max\big\{\sup_{x\in M}\|Df(x)\|,\sup_{x\in M}\|Df^{-1}(x)\| \big\}>1,\\
		\lambda_{\rm min}  \eqdef & \inf_{x\in M} \|Df|_{E^c(x)}\|.
	\end{split}
\end{equation}

\begin{Claim-numbered}\label{cl.nothappen}
	The center bundle of $\Lambda$ is not uniformly expanding.
\end{Claim-numbered}

\begin{proof}
	Assume, by contradiction, that the center is uniformly expanding and thus there is $\chi>0$ such that 
	\begin{equation}\label{eq:center-expansion-of-Lambda}
		\liminf_{n\rightarrow+\infty}\frac{1}{n}\log\|Df^n|_{E^c}\|>\chi \quad \textrm{for every $x\in\Lambda$}.
	\end{equation}

	By the uniform continuity of the bundles $E^s,E^c,E^u$ and the non-joint integrability of the strong stable and unstable bundles of the geodesic flow, there exist $\ell_0>0$  and   $\e_0>C^{-1}\e_0>\delta_0>0$ small such that 
	(as $f$ is fixed we omit the dependence of the foliations on it):
	\begin{itemize} 
		\item 
		for every $x\in M$ and $y\in\cF^s_{\e_0}(x)$, the {\em{strong stable holonomy map}} 
		\[
		{\rm H}_{y,x}^s\colon\cF^{cu}_{\ell_0}(y)\to\mathcal{F}^{cu}_{2\ell_0}(x),
		\qquad z\mapsto \cF^s_{\e_0}(z)\cap \mathcal{F}^{cu}_{2\ell_0}(x)
		\]
		is well defined;
		\item 
		for every $x\in M$,
		the local strong unstable leaf $\cF^u_{2\ell_0}(x)$ splits $\mathcal{F}^{cu}_{2\ell_0}(x)$ into two connected components $D^+$ and $D^-$;
		\item   
		for every $y\in \cF^s_{\e_0}(x)\setminus \cF^s_{\delta_0}(x)$, the set ${\rm H}^s_{y,x}(\cF^u_{\ell_0}(y))$ intersects both $D^+$ and $D^-$. 
	\end{itemize}
	The last item is a well known fact for the
	time-one map of contact Anosov flows, it is related to the nonintegrability of 
	$E^s\oplus E^u$.
	
	Let $L>1$ be the number  in the definition of  usu-accessibility for $f$.
	Let us denote 
	\[\cF^s(\Lambda) \eqdef\bigcup_{y\in\Lambda}\cF^s(y). \]
	\begin{SubClaim-numbered}\label{sc.twist-position}
		For every $x\in M$ and every $\e>0$,  there are a compact  unstable disk
		$D\subset \cF^u_{L+2\ell_0}(x)$ and  a point $w^\prime \in\Lambda$ such that for every $z\in D$, 
		$\cF^c_{\e}(z)$ intersects the set
		$\cF_{\e_0}^s(\cF^u_{\ell_0}(w^\prime))$
		transversely at a unique point $a_{z,w'} \ne z$.
		
	\end{SubClaim-numbered}
	\proof
	Fix $x\in M$ and $\e>0$.
	Given any point $w\in \Lambda$, using the usu-accessibility of $f$, 
	we get points $x^\prime\in\cF^u_{L}(x)$ and  $w^\prime\in\cF^u_L(w)$ such that $w^\prime\in\cF^s_L(x^\prime)$. By the uniform expansion along $E^u$ and uniform contraction along $E^s$, up to some forward or backward iterates, one can assume that  $w^\prime\in\cF^s_{\e_0}(x^\prime)\setminus \cF^s_{\delta_0}(x^\prime)$  
	and hence consider the strong stable holonomy map 
	\[
	{\rm H}^s_{w^\prime, x^\prime}: \mathcal{F}^{cu}_{\ell_0}(w^\prime)\to\mathcal{F}^{cu}_{2\ell_0}(x^\prime).
	\] 
	\begin{figure}[!htbp]
		\begin{center}
			\def\svgwidth{\columnwidth}
\begingroup%
  \makeatletter%
  \providecommand\color[2][]{%
    \errmessage{(Inkscape) Color is used for the text in Inkscape, but the package 'color.sty' is not loaded}%
    \renewcommand\color[2][]{}%
  }%
  \providecommand\transparent[1]{%
    \errmessage{(Inkscape) Transparency is used (non-zero) for the text in Inkscape, but the package 'transparent.sty' is not loaded}%
    \renewcommand\transparent[1]{}%
  }%
  \providecommand\rotatebox[2]{#2}%
  \newcommand*\fsize{\dimexpr\f@size pt\relax}%
  \newcommand*\lineheight[1]{\fontsize{\fsize}{#1\fsize}\selectfont}%
  \ifx\svgwidth\undefined%
    \setlength{\unitlength}{261.37361112bp}%
    \ifx\svgscale\undefined%
      \relax%
    \else%
      \setlength{\unitlength}{\unitlength * \real{\svgscale}}%
    \fi%
  \else%
    \setlength{\unitlength}{\svgwidth}%
  \fi%
  \global\let\svgwidth\undefined%
  \global\let\svgscale\undefined%
  \makeatother%
  \begin{picture}(1,0.32623512)%
    \lineheight{1}%
    \setlength\tabcolsep{0pt}%
    \put(0,0){\includegraphics[width=\unitlength,page=1]{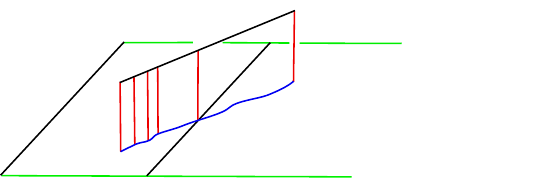}}%
    \put(0.30307831,0.0252861){\color[rgb]{0,0,0}\makebox(0,0)[lt]{\lineheight{1.25}\smash{\begin{tabular}[t]{l}$\mathcal{F}^u_{2\ell_0}(x^\prime)$\end{tabular}}}}%
    \put(0.54829544,0.30713528){\color[rgb]{0,0,0}\makebox(0,0)[lt]{\lineheight{1.25}\smash{\begin{tabular}[t]{l}$\mathcal{F}^u_{\ell_0}(w^\prime)$\end{tabular}}}}%
    \put(0.36406611,0.18925874){\color[rgb]{0,0,0}\makebox(0,0)[lt]{\lineheight{1.25}\smash{\begin{tabular}[t]{l}$\mathcal{F}^s_{\e_0}(x^\prime)$\end{tabular}}}}%
    \put(0,0){\includegraphics[width=\unitlength,page=2]{usu-accessible-1.pdf}}%
    \put(0.33807752,0.24109855){\color[rgb]{0.10196078,0.00392157,0.01960784}\makebox(0,0)[lt]{\lineheight{1.25}\smash{\begin{tabular}[t]{l}$w^\prime$\end{tabular}}}}%
    \put(0,0){\includegraphics[width=\unitlength,page=3]{usu-accessible-1.pdf}}%
    \put(0.37730651,0.08443925){\color[rgb]{0.10196078,0.00392157,0.01960784}\makebox(0,0)[lt]{\lineheight{1.25}\smash{\begin{tabular}[t]{l}$x^\prime$\end{tabular}}}}%
  \end{picture}%
\endgroup%

			\caption{Red lines denote strong stable plaques, black lines denote the strong unstable plaques and the green lines denote the center plaque.  }
		\end{center}
	\end{figure}
	Recall that ${\rm H}^s_{w^\prime, x^\prime}(\mathcal{F}^u_{\ell_0}(w^\prime))$ intersects the two connected components of $\mathcal{F}^{cu}_{2\ell_0}(x^\prime)\setminus\mathcal{F}^u_{2\ell_0}(x^\prime).$
	Since $\cF^{cs}_{\e_0}(w^\prime)\cap \mathcal{F}^u_{\ell_0}(w^\prime)\neq\emptyset$, by the uniform transversality between $\mathcal{F}^{cs}$ and $\mathcal{F}^u$ and the uniform continuity of the bundles $E^{cs}$ and $E^u$, there exists a compact  unstable  disc $D\subset\cF^u_{2\ell_0}(x^\prime)$ such that
	\begin{itemize}
		\item $D\cap {\rm H}^s_{w^\prime,x^\prime}(\cF^u_{\ell_0}(w^\prime))=\emptyset$;
		\item for every $z\in D$, one has $\cF^{cs}_{2\e_0}(z)\cap\mathcal{F}^u_{\ell_0}(w^\prime)\neq\emptyset$.
	\end{itemize}  
	The latter  item implies that 
	\[\cF^s_{2\e_0}(\mathcal{F}^u_{\ell_0}(w^\prime))\cap\cF^c_{2\e_0}(z)\neq\emptyset. \]
	Furthermore, this intersection is transverse and  consists of a single point that is different from $z$,  ending the proof of the subclaim.
	\endproof 
	
	For the next result recall the definition of  $B^c_{\chi} $ in \eqref{e.Bchi}.
	\begin{SubClaim-numbered}\label{sc.arbitrary-pliss-time}
		There exists a point  $z_0\in M$ such that 
		$$
		\#\left( \cF^c_{2 \ell_0} (z_0)\cap B^c_{\chi}  \right) =\infty.
		$$
	\end{SubClaim-numbered}
	
	Claim~\ref{cl.nothappen} follows immediately from this claim.
	Let $k_0$ be the integer given  by Proposition~\ref{p.finite-pliss} with respect to $\chi, 2\ell_0$.  By Subclaim~\ref{sc.arbitrary-pliss-time}, there exists a point $z_0\in  M$ satisfying  
	\[\#\bigg(\cF^c_{2\ell_0}(z)\cap B^c_\chi \Bigg) > k_0+1,
	\]
	contradicting Proposition~\ref{p.finite-pliss} and proving the lemma.  We now prove the claim.

	\begin{proof}[Proof of Subclaim~\ref{sc.arbitrary-pliss-time}]
		Fix a point $x_0\in M$. By  Subclaim~\ref{sc.twist-position}, there exist  a compact unstable curve
		$D_1\subset \cF^u_{L+2\ell_0}(x_0)$ and a point $w_1\in\Lambda$ such that for every $y\in D_1$, one has 
		$$
		\cF^c_{\ell_0}(y)\cap \cF_{\e_0}^s(\cF^u_{\ell_0}(w_1))= \{a_{y,w_1}\}.
		$$
		Denoting by
		$d^c$  the distance along center leaves,
		one has 
		\[\hat \e_1 \eqdef \inf_{y\in D_1}d^c\big(y, a_{y, w_1} \big)>0.
		\] 
		
		\begin{figure}[!htbp]
			\begin{center}
				\def\svgwidth{\columnwidth}
				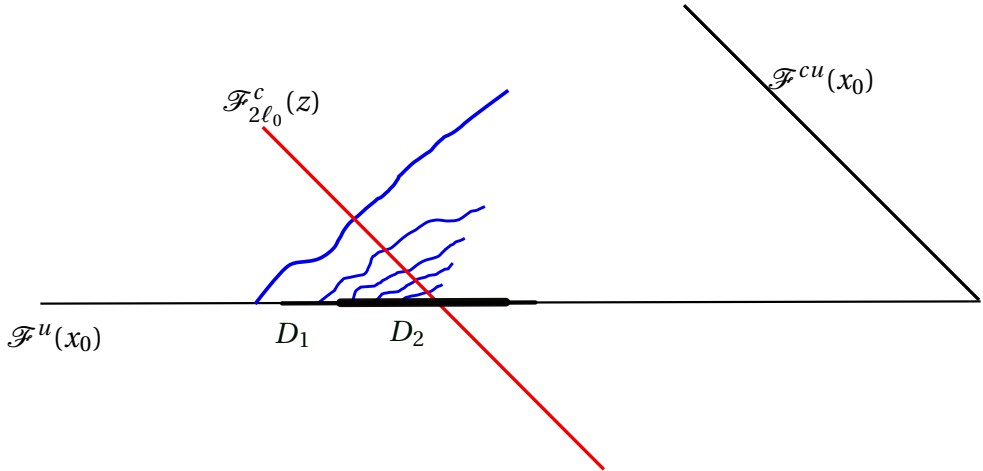
				\caption{Blue lines represents the intersection of $\mathcal{F}^{cu}(x_0)$ with $\mathcal{F}^s(\Lambda).$ }
			\end{center}
		\end{figure}

		By the uniform expansion of $f$ along strong unstable bundle, there exists $n_1\in\mathbb{N}$ such that $f^{n_1}(D_1)$ has diameter at least $L+2\ell_0$. Recall the definition of $\lambda_{\rm min}$ in \eqref{e.Dbounds} and define
		$$
		\e_1  \eqdef 2^{-1}\cdot\lambda_{\rm min}^{n_1}\cdot\hat\e_1.
		$$ 
		By Subclaim~\ref{sc.twist-position} and the comment above, there are  a compact unstable disc $\widehat D_2\subset f^{n_1}(D_1)$  and 
		a point $w_2\in\Lambda$ such that for every $y\in\widehat D_2$, one has 
		$$
		\cF^c_{\e_1}(y)\cap \cF_{\e_0}^s(\cF^u_{\ell_0}(w_2))= \{a_{y, w_2}\}
		\qquad \mbox{and} \qquad \hat \e_2 \eqdef \inf_{y\in D_1}d^c\big(y, a_{y, w_2} \big).
		$$ 
		By the choices of $\lambda_{\rm min}$ and $\e_1$, one has  
		$$ 
		a_{y, w_2} \not \in  f^{n_1}\big(\cF^c_{\ell_0}(f^{-n_1}(y))\cap \cF_{\e_0}^s(\cF^u_{\ell_0}(w_1))\big)
		$$
		and hence  $a_{y,w_2}\ne  a_{y,w_1}$.

		As the set $\Lambda$ is $\cF^u$-saturated and $w_1 \in \Lambda$, 
		taking $D_2=f^{-n_1}(\widehat D_2)\subset D_1$. One gets that 
		$$
		a_{y,w_2}, a_{y,w_1}\in \cF^c_{2\ell_0}(y)\cap \cF^s(\Lambda).
		$$ 
		Hence 
		\[
		\#\big(\cF^c_{2\ell_0}(y)\cap \cF^s(\Lambda)\big) \geq 2 \quad \textrm{for every $y\in D_2$}.
		\]
		Applying the above arguments inductively, one obtains a nested sequence of  compact unstable segments $(D_k)_{k\in\mathbb{N}}$
		such that  for every $y\in D_k$, one has 
		\[
		\#\big(\cF^c_{2\ell_0}(y)\cap \cF^s(\Lambda)\big) \geq k.
		\]
		
		By Equation~\eqref{eq:center-expansion-of-Lambda}, using the uniform contraction of $f$ along the strong stable foliation, one has 
		\[
		\liminf_{n\rightarrow+\infty}\frac{1}{n}\log\|Df^n|_{E^c(y)}\|\geq \chi\quad \textrm{for every $y\in\mathcal{F}^s(\Lambda)$}.
		\]
		Now considering any point $z_0\in \bigcap_{k\in\mathbb{N}}D_k$ the subclaim follows. 
	\end{proof} 
	
	The proof of Claim~\ref{cl.nothappen} is now complete
\end{proof}

As Theorem~\ref{thm.geodesic-flow-case} follows from Claim~\ref{cl.nothappen}, its proof is now complete.
\hfill \qed

\section{Skew-products: Proofs of Theorems~\ref{thm.nilmanifold-case} and Corollary~\ref{ctct.both-minimal}} \label{s.skewcircle}

Before proving Theorem~\ref{thm.nilmanifold-case} in Section~\ref{ss.proofsk},
we will first collect 
some results about measures of maximal entropy in Section~\ref{ss.maximalent}.
In Section \ref{ss.bothminim}, we  discuss the case where both strong foliations are minimal and prove 
Corollary~\ref{ctct.both-minimal}.

\subsection{Measures of maximal entropy for skew-products}
\label{ss.maximalent}

We now recall some properties of measures of maximal entropy for partially hyperbolic skew-products with circle fibers. A partially hyperbolic skew-product with a circle fiber is said to be of \emph{rotation type} if it is $C^0$-conjugate to a rotation extension of an Anosov diffeomorphism on the torus. It was shown in \cite{Rodetal:12} that every
accessible partially hyperbolic $C^2$-diffeomorphism on a three-manifold with one dimensional center bundle,
where the center foliation forms a circle bundle,  is either of rotation type or possesses finitely many ergodic measures of maximal entropy. Subsequent improvements to this result can be found in \cite{TahYan:19,RocTah:22,UreViYan:21}. Here, we use the following characterization of measures of maximal entropy for skew-products, as derived from \cite{UreViYan:21}.

\begin{Theorem}[Theorem C in \cite{UreViYan:21}]\label{thm.skew-product-mme}
	Let $f\in\ph^2_{c=1}(M)$ be accessible and dynamically coherent. Assume that  the center foliation $\cF^c$ forms a circle bundle and the quotient space $M^c = M/\cF^c$ is a torus. 
	Then the following dichotomy holds: 
	\begin{itemize}
		\item Either there is a unique measure of maximal entropy and this 
		measure is nonhyperbolic and has full support;
		\item or there are  finitely many ergodic measures of maximal entropy 
		$$
		\mu_1,\cdots,\mu_\ell, \nu_1,\cdots,\nu_m, \quad
		\ell, m\ge 1,
		$$ 
		such that 
		\begin{itemize}
			\item[--] every $\mu_i$ has positive center Lyapunov exponent 
			and $\supp(\mu_i)$ has finitely many connected components, each of which is a $\cF^s$-minimal set;
			\item[--]  every $\nu_i$ has negative center Lyapunov exponent and $\supp(\nu_i)$ has finitely many connected components, each of which is a $\cF^u$-minimal set.
		\end{itemize}
	\end{itemize}
	If the first case of the dichotomy holds then  $f$ is of rotation type.
\end{Theorem}

\begin{Remark}\label{r.skew-product-mme}
	{\em{	In the second case of the dichotomy in Theorem~\ref{thm.skew-product-mme}, since the support of every ergodic measure $\mu_i$ is $f$-invariant, each connected component of $\supp(\mu_i)$ is $f^{k_i}$-invariant for some $k_i \in \mathbb{N}$, leading to an $f^{k_i}$-invariant and $\mathcal{F}^s$-minimal set $\Lambda^s_i$. Moreover, 
			using partial hyperbolicity, if follows from \cite{Ga:02, Cro:11, Gel:16} that
			$\supp(\mu_i)$ contains periodic points with positive center Lyapunov exponents, so does $\Lambda^s_i$. Similarly, there are $f^{n_j}$-invariant and $\mathcal{F}^u$-minimal sets  $\Lambda_j^u$ which contains periodic points with negative center Lyapunov exponent.}}
\end{Remark}

\subsection{Proof of Theorem~\ref{thm.nilmanifold-case}}
\label{ss.proofsk}

We are now ready to prove Theorem~\ref{thm.nilmanifold-case}.
By Theorem~\ref{thm.skew-product-mme} and Remark~\ref{r.skew-product-mme}, one has the following dichotomy
\begin{itemize}
	\item either $f$ has a unique measure of maximal entropy which is nonhyperbolic;
	\item or 
	$f$ has  $\ell$-ergodic measures $\mu$ of maximal entropy with  $\lambda^c(\mu)>0$ whose supports are $\mathcal{F}^s$-minimal sets that contain  points $p\in\mathrm{Per}(f)$ with $\lambda^c(p)>0$ and
	$m$-ergodic measures $\nu$ of maximal entropy with $\lambda^c(\nu)<0$ whose supports are $\mathcal{F}^u$-minimal sets that contain points $p\in\mathrm{Per}(f)$ with $\lambda^c(p)<0$,
\end{itemize}
The first case gives rise to a nonhyperbolic  measure, and we are done. Thus it is enough consider the second case.

Note that
since $f$ is accessible,  so is $f^k$. Then the volume measure $\mathrm{Leb}$ is $f^k$ and $f$-ergodic due 
Theorem~\ref{t.thenitisergodic}. If $\mathrm{Leb}$ has zero center Lyapunov exponent for $f$, then  $\mathrm{Leb}$ is nonhyperbolic 
and we are done. 

It remains to consider the case when $\mathrm{Leb}$ is hyperbolic.
As the measure $\mathrm{Leb}$ is $f^k$-ergodic for every $k$. Our assumption implies that it also hyperbolic for $f^k$.
By the comment above,  $f^k$ also has hyperbolic  measures of maximal entropy. 

By Remark~\ref{r.skew-product-mme},  there exists $k\in\mathbb{N}$ such that there exist an $f^k$-invariant $\mathcal{F}^u$-minimal set $\Lambda^u$ and an $f^k$-invariant $\mathcal{F}^s$-minimal set $\Lambda^s$. Furthermore, $\Lambda^u$ contains points
$p\in\mathrm{Per}(f)$ with $\lambda^c(p)<0$
and $\Lambda^s$ contains points $p\in\mathrm{Per}(f)$ with $\lambda^c(p)>0$.
Up to considering $f^{-k}$, one can assume that $\lambda^c(\mathrm{Leb}, f^k)>0$ (considered with respect to $f^k$).
Thus $f^k$ is mostly expanding due to Theorem~\ref{thm.mostly-expanding-criterion}. 	 
Now $\Lambda^u$ satisfies the assumption of Theorem~\ref{thm.existence}, thus $f^k$ has a nonhyperbolic  measure, which is also a nonhyperbolic measure for $f$.  The proof of the theorem is now complete.
\hfill \qed
\endproof

\subsection{Proof of Corollary~\ref{ctct.both-minimal}}
\label{ss.bothminim}
Note that if $f$ or $f^{-1}$ has some nonhyperbolic $u$-Gibbs state, then we obtain nonhyperbolic measures.
Thus, from now on, we  can assume that all ergodic $u$-Gibbs states for $f$ and all ergodic $u$-Gibbs states for $f^{-1}$ are hyperbolic. We have two cases to consider:
\begin{itemize}
	\item[(i)]  either  every ergodic  $u$-Gibbs state $\mu$ for $f$
	satisfies $\lambda^c(\mu)>0$;
	\item[(ii)] or  there exists some ergodic  $u$-Gibbs state $\mu$  with
	$\lambda^c(\mu) <0$. 
\end{itemize} 

Consider first the case (i), where $f$ is by definition mostly expanding. Thus
we are precisely in the setting of Proposition~\ref{p.intermediate} 
for $f$ with $\Lambda=M$. Therefore  
\begin{itemize}
	\item either $f$ has  nonhyperbolic measures;
	\item  or $f$ is Anosov with uniformly expanding center,
\end{itemize}
proving the theorem in this case.

It remains to consider the case (ii). 
As we are considering forward and backward dynamics, for the sake of clearness, we will include $f$ or $f^{-1}$ in the notation of the exponents.

\begin{Claim-numbered}
	\label{cl.onlyone}
	Let $\mu\in \Gib (f)$  with $\lambda^c(\mu,f)<0$.
	Then  $\mu=\mathrm{Leb}$. Moreover, every ergodic $u$-Gibbs state $\nu$ of $f^{-1}$ satisfies $\lambda^c(\nu, f^{-1})>0$.
\end{Claim-numbered}

Note that in this case $f^{-1}$ is mostly expanding. As the strong stable foliation of $f$ is minimal, the hypotheses of Proposition~\ref{p.intermediate} 
hold for $f^{-1}$ with $\Lambda=M$. Thus either there exists some nonhyperbolic measure or $f$ is Anosov with uniformly contracting center, proving the theorem.

Let us prove the claim.

\begin{proof}[Proof of Claim~\ref{cl.onlyone}]
	As $\mu\in \Gib(f)$ with $\lambda^c(\mu, f)<0$, 
	by the minimality of the strong unstable foliation and the absolute continuity of the Pesin's stable manifolds \cite{Pes:77}, $\mu$ is the unique $u$-Gibbs state of $f$. By the absolute continuity of the strong unstable foliation, the volume measure 
	$\mathrm{Leb}$ is a $u$-Gibbs state, and thus $\mu=\mathrm{Leb}.$ Note that $\lambda^c(\mathrm{Leb}, f)<0$.
	
	To prove the second part, arguing as in the first part, if there is a $u$-Gibbs state $\nu$ of $f^{-1}$ with 
	$\lambda^c(\nu, f^{-1})<0$ then
	$\nu=\mathrm{Leb}$ and hence $\lambda^c(\mathrm{Leb}, f^{-1})<0$, that is $\lambda^c (\mathrm{Leb}, f)>0$, a contradiction. This ends the proof of the claim.
\end{proof}

\bibliographystyle{alpha}

\newcommand{\etalchar}[1]{$^{#1}$}

 
\vspace{2mm}
 
 \begin{tabular}{l l l}
	\emph{\normalsize Lorenzo J. D\'iaz}
	& \quad &
	\emph{\normalsize Jiagang Yang}
	\medskip\\ 	  
	\small  Departamento de Matem\'atica 
	&& \small Departamento de Geometria, 
	\\
	\small PUC-Rio 
	&&	\small Instituto de Matem\'atica e Estat\'istica, 
	 \\ 
	\small 	Rio de Janeiro 22451-900, Brazil
		&& \small  Universidade
	Federal Fluminense\\
	\small \texttt{lodiaz@puc-rio.br}
	&& \small  Niter\'oi,  24020-140, Brazil\\
	&& \small \texttt{jiagangyang@id.uff.br}\\
	&& \small
	\bigskip\\
	\emph{\normalsize Jinhua Zhang}
	\medskip\\
	\small School of Mathematical Sciences\\ 
	\small Beihang University\\ 
	\small Beijing, 100191, P. R.  China\\ 
	\small \texttt{jinhua$\_$zhang@buaa.edu.cn}\\
	\small \texttt{zjh200889@gmail.com}
	
\end{tabular}

\end{document}